\documentclass [11pt]{amsart}

\usepackage{amscd,amsmath,amssymb,euscript,tikz,todonotes}
\usepackage[frame,cmtip,curve,arrow,matrix,line,graph]{xy}

\date{}

\title{Scattering diagrams, Hall algebras and stability conditions}
\author{Tom Bridgeland}

\newtheorem{thm}{Theorem}[section]

\newtheorem{prop}[thm]{Proposition}
\newtheorem{lemma}[thm]{Lemma}

\theoremstyle{definition}
\newtheorem{defn}[thm]{Definition}

\newtheorem{thm*}[thm]{Theorem$^*$}

\newtheorem*{example*}{Example}

\newcommand {\A}{\mathcal A}
\newcommand{\T}{\mathcal{T}}
\newcommand{\F}{\mathcal{F}}
\newcommand{\W}{\mathcal{W}}
\renewcommand{\L}{\mathbb{L}}

\newcommand{\R}{\mathbb{R}}
\newcommand{\Q}{\mathbb{Q}}
\newcommand{\Z}{\mathbb{Z}}
\newcommand{\D}{\mathcal{D}}

\newcommand {\C}{\mathbb C}
\newcommand {\Hom}{\operatorname{Hom}}

\newcommand {\Ext}{\operatorname{Ext}}
\newcommand {\Rep}{\operatorname{Rep}}

\newcommand {\id}{\operatorname{id}}
\newcommand{\mat}[4]{\begin{pmatrix}#1&#2\\#3&#4\end{pmatrix}}

\newcommand{\DD}{\mathfrak{D}}
\newcommand{\dd}{\mathfrak{d}}

\renewcommand{\ss}{{\operatorname{ss}}}
\renewcommand{\leq}{\leqslant}
\renewcommand{\geq}{\geqslant}

\newcommand{\onto}{\twoheadrightarrow}
\newcommand{\lra}{\longrightarrow}

\newcommand {\<}{\langle}
\renewcommand {\>}{\rangle}
\newcommand {\Omit}[1]{}
\newcommand{\isom}{\cong}
\newcommand{\halft}{{\tfrac{1}{2}}}

\newcommand{\tensor}{\otimes}

\newcommand{\St}{\operatorname{St}}
\newcommand{\reg}{{\operatorname{reg}}}

\newcommand{\Hall}{{\operatorname{Hall}}}

\renewcommand{\P}{\mathcal{P}}
\newcommand{\CC}{\mathcal{C}}
\newcommand{\g}{\mathfrak{g}}

\newcommand{\h}{\mathfrak{h}}

\newcommand{\blank}{\kern.0mm\smash{-}\kern.0mm}

\newcommand{\B}{\mathcal{B}}

\newcommand{\lRa}[1]{\stackrel{#1}{\lra}}

\newcommand{\coker}{\operatorname{coker}}

\newcommand{\Stab}{\operatorname{Stab}}
\renewcommand{\Im}{\operatorname{Im}}

\newcommand{\M}{\mathcal{M}}
\newcommand{\esupp}{\operatorname{supp}_{\operatorname{ess}}}
\newcommand{\one}{\mathbf{1}}

\renewcommand{\mod}{\operatorname{mod}}

\newcommand{\rep}{\operatorname{rep}}

\newcommand{\CY}{CY$_3$\ }
\newcommand{\an}{\operatorname{an}}
\newcommand{\Spec}{\operatorname{Spec}}
\renewcommand{\sc}{{\operatorname{sc}}}
\newcommand{\fr}{\operatorname{rep}}
\newcommand{\mfr}{\operatorname{\mathcal{M}}}

\newcommand{\N}{\mathcal{N}}

\newcommand{\E}{\mathcal{E}}
\newcommand{\End}{\operatorname{End}}

\newcommand{\GL}{\operatorname{GL}}
\newcommand{\kst}[1]{K({\operatorname{St/{#1}}})}
\newcommand{\supp}{\operatorname{supp}}
\renewcommand{\S}{\mathfrak{S}}

\renewcommand{\O}{\mathcal{O}}
\newcommand{\CQ}{\C Q}

\newcommand{\grp}{{\widetilde{\GL^+}(2,\R)}}

\newcommand{\V}{\mathcal{V}}
\begin{document}
\begin{abstract}{To any quiver with relations we associate a consistent scattering diagram  
taking values in the motivic Hall algebra of its category of representations. We show that the chamber structure of this scattering diagram coincides with the natural chamber structure in an open subset of the space of stability conditions on the associated triangulated category. In the three-dimensional Calabi-Yau situation, when the relations arise from a potential,  we  can apply an integration map to give a consistent scattering diagram taking values in a tropical vertex group. }\end{abstract}
\dedicatory
{In memory of Kentaro Nagao}

\email{t.bridgeland@sheffield.ac.uk}
\address{School of Mathematics and Statistics, University of Sheffield, Hicks Building, Hounsfield Road, S3 7RH, United Kingdom}

\maketitle

\section{Introduction}

The concept of a scattering diagram has emerged from the work of Kontsevich and Soibelman \cite{ks1}, and Gross and Siebert \cite{gs}, on the Strominger-Yau-Zaslow approach to mirror symmetry.  Scattering diagrams, and the associated combinatorics of broken lines, have  been used  recently by Gross, Hacking, Keel and Kontsevich \cite{ghkk} to settle several important conjectures about cluster algebras. The aim of this paper is to explain some very general  connections between spaces of stability conditions and scattering diagrams. In particular, we explain  the relationship between the scattering diagram associated to an acyclic quiver $Q$, and the wall-and-chamber structure of the  space of stability conditions  on the associated \CY triangulated category. Our longer-term goal is to better understand the geometrical relationship between cluster varieties and spaces of stability conditions. 

It turns out that to give a categorical description of the scattering diagram    one does not need to use triangulated categories. In fact, for the most part, we work in  the abelian category  of representations of a fixed quiver with relations $(Q,I)$, and use King's notion of $\theta$--stability. Our starting point is the observation that the walls in the scattering diagram  correspond precisely to the values of the weight vector $\theta$ for which there exist nonzero $\theta$--semistable representations.

To fully specify the  scattering diagram we must first fix a graded Lie algebra. For the most general and abstract form of our result this is defined using the motivic Hall algebra \cite{bridgeland2,joyce1,js,ks2} of the category of representations of $(Q,I)$. In the Calabi-Yau threefold case, when the relations  arise from a potential,  we can then apply an integration map  to produce a scattering diagram taking values in a much simpler Lie algebra  spanned as a vector space by the dimension vectors of non-zero representations of $Q$. We feel that the story becomes clearer by working first in the Hall algebra and using the \CY assumption only when absolutely necessary.

Much of what we do here builds on earlier work of Reineke \cite{reineke},   Gross and Pandharipande \cite{gp},  Kontsevich and Soibelman \cite{ks1,ks2}, Nagao \cite{nagao}, Keller \cite{keller} and Gross, Hacking, Keel and Kontsevich \cite{ghkk}.
The main novelties in our approach are the introduction of the Hall algebra scattering diagram associated to an arbitrary quivers with relations,  and the link with spaces of stability conditions on triangulated categories explained in Section \ref{stab}.  Although simple,  this second point is likely to be crucial in unravelling the connection between spaces of stability conditions and cluster varieties. 

The rest of the Introduction contains a more detailed description of our results. For simplicity we will only consider here the scattering diagrams associated to quivers with potential taking values in the Lie algebra of functions on a Poisson torus, leaving the more general construction of Hall algebra scattering diagrams to the body of the paper.

\subsection{Scattering diagrams}

We begin by giving a rough idea of the notion of a scattering diagram: for further details the reader is referred to  Section \ref{scatter}.
Fix a finitely-generated free abelian group $N$ and set \[M=\Hom_\Z(N,\Z).\] Let us also fix a $\Z$-basis  $(e_1,\cdots,e_k)$ of $N$ and define
$N^+\subset N$ to be the cone consisting of nonzero elements of $N$ which can be written as  non-negative integral combinations of the basis elements.
Consider a  graded Lie algebra
$\g=\bigoplus_{n\in N^+} \g_n$ and the  pro-nilpotent Lie algebra 
\[\Hat{\g}=\prod_{n\in N^+} \g_n \]
  obtained by completing $\g$ with respect to the grading. There is  an associated pro-unipotent algebraic group $\Hat{G}$
with a bijective exponential map \[\exp\colon \Hat{\g}\to \Hat{G}\] defined formally by the Baker--Campbell--Hausdorff formula.

Given this data, a scattering diagram $\DD$  consists of a collection of codimension one closed subsets \[\dd\subset M_\R=M\tensor_\Z \R\]  known as  walls, together with associated elements $\Phi_\DD(\dd)\in \Hat{G}$. Each wall is a convex cone in the hyperplane $n^\perp\subset M_\R$
defined by some primitive vector $n\in N^+$,
and the corresponding element $\Phi_\DD(\dd)\in \Hat{G}$  is then required to lie in the subgroup \[\exp\big(\bigoplus_{k\geq 0} \g_{kn}\big)\subset \Hat{G}.\]  The support $\supp(\DD)\subset M_\R$ of a scattering diagram $\DD$ is defined to be the union  of its walls.

 A scattering diagram $\DD$ is called consistent if for any sufficiently general path $\gamma\colon [0,1]\to M_\R$, the ordered product
\[\Phi_\DD(\gamma)=\big.\prod \Phi_\DD(\dd_i)^{\pm 1}\in \Hat{G},\]
corresponding to the sequence of walls $\dd_i$ crossed by $\gamma$, depends only on the endpoints of $\gamma$.
Two scattering diagrams $\DD_1$ and $\DD_2$  taking values in the same graded Lie algebra $\g$ are called equivalent if  one has
\[\Phi_{\DD_1}(\gamma)=\Phi_{\DD_2}(\gamma),\]
for any sufficiently general path $\gamma\colon [0,1]\to M_\R$. An important fact proved by Kontsevich and Soibelman \cite{ks1} is that equivalence classes of consistent scattering diagrams are in bijection with elements of the group $\Hat{G}$. We explain the proof of this in Section 3.

\subsection{Stability scattering diagram}
Let $Q$ be a  quiver without vertex loops or oriented 2-cycles. We call such quivers $2$-acyclic. Let $V(Q)$ denote the set of vertices of $Q$. We set $N=\Z^{V(Q)}$ and denote the canonical basis by $(e_i)_{i\in V(Q)}\subset N$.  We write $N^+\subset N$ for the corresponding positive cone as above, and also set $N^{\oplus}=N\cup\{0\}$.

The negative of the skew-symmetrized adjacency matrix of $Q$ defines a skew-symmetric form
\[\<-,-\>\colon N\times N\to \Z.\]
The monoid algebra $\C[N^\oplus]$ then becomes a Poisson algebra with bracket
\[\{x^{n_1},x^{n_2}\}=\<n_1,n_2\> \cdot x^{n_1+n_2}.\]
We define the Lie algebra $\g$ to be the subspace $\C[N^+]$ equipped with this Poisson  bracket.

Suppose now that our quiver $Q$ is equipped with a finite  potential $W\in \C Q$. By definition $W$ is a finite linear combination of cycles in $Q$.    Let $J(Q,W)$  be the  corresponding Jacobi algebra, and let \[\A=\mod J(Q,W)\]
be its category of finite-dimensional  left modules. 
Any  object $E\in \A$ has an associated dimension vector $d(E)\in N^\oplus$. 
Given a vector $\theta\in M_\R$ we always write $\theta(E)$ in place of the more cumbersome $\theta(d(E))$.
An object  $E\in \A$ is said to be $\theta$-semistable if
\begin{itemize}
\item[(i)] $\theta(E)=0$,
\item[(ii)]  every subobject $A\subset E$ satisfies $\theta(A)\leq 0$.
\end{itemize} 

King \cite{king} proved that there is a  quasi-projective moduli scheme $M(d,\theta)$ parameterizing $\theta$--stable representations of $J(Q,W)$ of dimension vector $d$.
Joyce \cite{joyce4} showed how to define associated rational numbers $J(d,\theta)\in \Q$ which we call Joyce invariants. We will explain how to define these invariants in Section 11: they are uniquely determined by their wall-crossing properties, together with the formula
\[J(d,\theta)=e(M(d,\theta))\in \Z,\]
which holds when $d\in N^\oplus$ is primitive and $\theta\in d^\perp$ is general. Here   $e(X)$ denotes the Euler number of a complex variety $X$ equipped with the classical topology.

\begin{thm}
\label{main}
Let $Q$ be a 2-acyclic quiver with finite potential $W$. There
 is a  consistent scattering diagram $\DD$ taking values in the Lie algebra $\g$ such that \begin{itemize}
\item[(a)]
the support 
of $\DD$  consists of those maps $\theta\in M_\R$ for which there exist nonzero $\theta$-semistable objects in $\A$,
\item[(b)]  the wall-crossing automorphism at a general point of the support of $\DD$ is
\[\Phi_\DD(\dd)=\exp\bigg(\sum_{d\in\theta^\perp} J(d,\theta) \cdot x^d\bigg)\in \Hat{G}.\]
\end{itemize}
\end{thm}

We refer to the scattering diagram of Theorem \ref{main} as the \emph{stability scattering diagram} for the pair $(Q,W)$. It is unique up to equivalence.

\begin{figure}{\center
\begin{tikzpicture}
\draw (0,0) --(2,0);
\node at (2,0.3){$\Phi_{(0,1)}$};
\node at (-2,0.3){$\Phi_{(0,1)}$};
\node at (0.6,1.7){$\Phi_{(1,0)}$};
\node at (0.6,-1.8){$\Phi_{(1,0)}$};
\node at (1.7,-1.1){$\Phi_{(1,1)}$};
\draw (0,0) --(0,2);
\draw (0,0) --(-2,0);
\draw (0,0) --(0,-2);
\draw (0,0) --(2,-2);
\draw[->][thick] ([shift=(75:1.5)]0,0) arc (75:105:1.5);
\draw[->][thick] ([shift=(285:1.5)]0,0) arc (285:255:1.5);
\draw[->][thick] ([shift=(330:1.5)]0,0) arc (330:300:1.5);
\draw[->][thick] ([shift=(165:1.5)]0,0) arc (165:195:1.5);
\draw[->][thick] ([shift=(15:1.5)]0,0) arc (15:-15:1.5);
\end{tikzpicture}
\caption{The scattering diagram for the $A_2$ quiver: there are three  walls corresponding  to the three indecomposable representations.}
\label{figone}}

\bigskip

{\center\begin{tikzpicture}
\node at (0.6,1.7){$\Phi_{(1,0)}$};
\node at (-2,0.3){$\Phi_{(0,1)}$};
\draw[->][thick] ([shift=(75:1.5)]0,0) arc (75:105:1.5);
\draw[->][thick] ([shift=(165:1.5)]0,0) arc (165:195:1.5);
\draw (0,0) --(2,0);
\draw (0,0) --(0,2);
\draw (0,0) --(-2,0);
\draw (0,0) --(0,-2);
\draw[thick] (0,0) --(2,-2);
\draw (0,0) --(2,-1);
\draw (0,0) --(1,-2);
\draw (0,0) --(1.33,-2);
\draw (0,0) --(2,-1.33);

\draw (0,0) --(1.5,-2);
\draw (0,0) --(2,-1.5);
\draw (0,0) --(1.6,-2);
\draw (0,0) --(2,-1.6);

\draw (0,0) --(1.67,-2);
\draw (0,0) --(2,-1.67);

\node at (1.75,-2){.};
\node at (1.85,-2){.};
\node at (2,-1.75){.};
\node at (2,-1.85){.};

\end{tikzpicture}
\caption{The scattering diagram for the Kronecker quiver: the walls correspond to the dimension vectors of the indecomposable representations.}
\label{figtwo}}

\bigskip{\center 
\begin{tikzpicture}
\node at (0.6,1.7){$\Phi_{(1,0)}$};
\node at (-2,0.3){$\Phi_{(0,1)}$};
\draw[->][thick] ([shift=(75:1.5)]0,0) arc (75:105:1.5);
\draw[->][thick] ([shift=(165:1.5)]0,0) arc (165:195:1.5);
\draw (0,0) --(2,0);
\draw (0,0) --(0,2);
\draw (0,0) --(-2,0);
\draw (0,0) --(0,-2);

\draw (0,0) --(2,-.66);
\draw (0,0) --(.66,-2);

\draw (0,0) --(0.47,-2);
\draw (0,0) --(2,-0.47);

\draw (0,0) --(0.76,-2);
\draw (0,0) --(2,-0.76);
\node at (1.6,-2){.};
\node at (1.7,-2){.};
\node at (2,-1.6){.};
\node at (2,-1.7){.};

\draw [black, fill=black] (0,0) -- (0.76,-2) -- (2,-2)--(2,-0.76) -- (0,0);
\end{tikzpicture}
\caption{The scattering diagram for a generalized Kronecker quiver with 3 arrows: there is a region in which the walls are  dense.}
\label{figthree}}

\end{figure}

\subsection{Examples: Kronecker quivers}

As an illustration let us consider the case of a generalized Kronecker quiver $Q$ with two vertices $V(Q)=\{1,2\}$ and  $p\geq 1$ arrows, all going from vertex $1$ to vertex $2$.  We necessarily have $W=0$ and the category $\A$ is just the usual category of finite-dimensional representations of $Q$. The resulting stability scattering diagrams are of course well-known (see for example \cite{gp}).

The lattice $N=\Z^{\oplus 2}$ has a basis $e_1,e_2$ indexed by the vertices of $Q$. Taking dimension vectors gives an identification $d\colon K_0(\A)\to N$, under which  the class $[S_i]\in K_0(\A)$ of the  simple representation at the vertex $i$ is mapped to the basis vector $e_i$.
The stability scattering diagram then lives in $M_\R=\R^2$ with co-ordinates $y_i=\theta(S_i)$. For $1\leq p\leq 3$ this scattering diagram is illustrated in Figure $p$.

The Lie algebra $\g$ is the ideal $(x_1,x_2)\subset \C[x_1,x_2]$ equipped with the restriction of the Poisson bracket
\[\<x_1,x_2\>=-k x_1 x_2.\]
Similarly  $\Hat{\g}=(x_1,x_2)\subset \C[[x_1,x_2]]$.
Given a vector $d\in N$, define an  automorphism of $\C[[x_1,x_2]]$ preserving the ideal $\Hat{\g}$ by the formula
\[\Phi_d(x^n)=x^n\cdot (1+x^d)^{\<d,n\>}\]

Consider first the case $p=1$ when $Q$ is the  A$_2$ quiver. This has finite representation type, and the category $\A$ has three indecomposable objects: $S_1$, $S_2$ and an extension
\[0\lra S_2\lra E\lra S_1\lra 0.\]
The scattering diagram has a wall for each of these objects. The object $S_1$ is $\theta$-stable precisely if $y_1=\theta(S_1)=0$: this is the vertical axis on the diagram. Similarly the object $S_2$ is $\theta$-stable on the horizontal axis $y_2=0$. The object $E$ is $\theta$-stable precisely if $\theta(E)=0$ and $\theta(S_2)<0$: this is the wall $y_1+y_2=0$ and $y_2<0$.
The diagram is consistent because of the pentagon identity
\[\Phi_{(0,1)}\circ \Phi_{(1,0)}=\Phi_{(1,0)}\circ \Phi_{(1,1)}\circ \Phi_{(0,1)}.\]

The case $p=2$ is the original Kronecker quiver. This is of tame representation type and has infinitely many indecomposable representations: unique indecomposables of dimension vector $(n,n-1)$ and $(n-1,n)$ for each $n\geq 1$ and a $\mathbb{P}^1$ family of indecomposables for each dimension vector $(n,n)$. Each of these dimension vectors defines a wall of the scattering diagram. The basic relation in this case is
\[\Phi_{(0,1)}\circ \Phi_{(1,0)}=\Phi_{(1,0)}\circ \Phi_{(2,1)}\circ \Phi_{(3,2)}\circ\cdots \Phi_{(1,1)}^{-2}\circ \cdots\circ \Phi_{(2,3)}\circ \Phi_{(1,2)}\circ \Phi_{(0,1)}.\]
Note that the right hand side makes perfect sense as an automorphism of $\C[[x_1,x_2]]$ because for any  $k>0$ only finitely many terms act non-trivially on the quotient $\C[x_1,x_2]/(x_1,x_2)^k$.

For $p\geq 3$ the quiver $Q$ is of wild representation type and the picture becomes much more complicated. In particular, there is a region of the scattering diagram in which the walls are dense.

\subsection{Stability conditions and walls of type II}

The scattering diagram we associate to a quiver with potential is closely related to the wall-and-chamber structure on the space of stability conditions of the corresponding \CY triangulated  category. 

Let $(Q,W)$ be a 2-acyclic quiver with finite potential. For simplicity we shall assume that the Jacobi algebra $J(Q,W)$ is finite-dimensional.
Let $\D$ denote the bounded derived category of the Ginzburg algebra of the pair $(Q,W)$. It is a \CY triangulated category with a distinguished bounded t-structure whose heart
 is  equivalent to the abelian category $\A=\mod J(Q,W)$. 

We let $\Stab(\D)$ denote the space of stability conditions on $\D$ satisfying the support property. It is a complex manifold. The forgetful map
\[\mathcal{Z}\colon \Stab(\D)\to M_\C\]
sending a stability condition $(Z,\P)$ to its central charge $Z\colon N\to \C$ is a local homeomorphism. There is a wall-and-chamber structure (although in general the walls are dense) such that the heart of the stability condition is constant in a given chamber. These walls are usually known as walls of type II to distinguish them from walls  where objects of a given fixed class can become stable or unstable.

We consider the open subset $\CC(\A)\subset \Stab(\D)$ of nearby stability conditions: by definition these are stability conditions for which all Harder-Narasimhan factors of objects of $\A$ have phases in the open interval $(-1,1)$.

\begin{thm}
The  map $\Im \mathcal{Z}\colon \CC(\A)\to M_\R$ is surjective. Moreover
\begin{itemize}
\item[(a)]
stability conditions in the same fibre of $\Im \mathcal{Z}$ have the same heart; 
\item[(c)] the support of the stability scattering diagram is precisely the image of the union of all type II walls in $\CC(\A)$ under the map $\Im \mathcal{Z}$.\end{itemize}
\end{thm}

Thus in each connected component of the complement of the closure of the support of the stability scattering diagram there is a well-defined  heart in $\D$. This heart is easily seen to be of finite-length, and the  classes of its simple objects form a basis of  $N=K_0(\D)$. Viewed in terms of the original basis $(e_i)$, these are precisely the $c$-vectors of cluster theory. As pointed out by Nagao \cite{nagao}, from this point of view the sign-coherence condition of c-vectors is immediate.

\subsection{Framed quiver moduli }
\label{fr}

Instead of using the Joyce invariants $J(d,\theta)$ we can also describe the stability scattering diagram using  Euler numbers of moduli spaces of framed quiver representations. These spaces are  generalizations of those studied by Engel and Reineke \cite{er}.

Fix a class $m\in M^+$ and form a new quiver $Q^\star$  extending $Q$, by adjoining a new vertex $\star$  and adding $m(e_i)$ arrows from vertex $\star$ to vertex $i$. The potential $W$ induces a potential  on $Q^\star$ in the obvious way. We let \[N^\star=\Z^{V(Q^\star)}=\Z^{V(Q)}\oplus \Z=N\oplus \Z, \quad M^\star=\Hom_\Z(N^\star,\Z)=M\oplus\Z.\]
 Given a dimension vector $d\in N$, we define $d^\star=(d,1)\in N^\star$. Given $\theta\in M_\R$ and $d\in N$ there is a unique lift $\theta^\star\in M^\star_\R$ such that $\theta^\star(d^\star)=0$.

There is a coarse moduli scheme $M^\ss(d^\star,\theta^\star)$ for $\theta^\star$-semistable representations of the Jacobi algebra $J(Q^\star, W^\star)$. This moduli scheme is fine providing that $\theta$ does not lie on one of the finitely many hyperplanes $n^\perp$ for dimension vectors $n\in N^+$ of smaller total dimension than $d$. For arbitrary $\theta\in M_\R$ we set
\[F(d,m,\theta):=M^\ss(d^\star,(\theta-\epsilon\delta)^\star)\quad 0<\epsilon\ll 1.\]
When  $W=0$ and $\theta(d)=0$ this moduli space   coincides with the smooth quiver moduli space of  \cite{er}. We denote its Euler characteristic by $K(d,m,\theta)$.

\begin{thm}
\label{aut}
Let $(Q,W)$ be a 2-acyclic quiver with finite potential and $\DD$ the corresponding stability scattering diagram. Then
the adjoint action of   $\Phi_\DD(\dd)\in G$  at a general point of the support of $\DD$ is
\[x^n\mapsto x^n \cdot \prod_{i\in V(Q)}\bigg(\sum_{d\in \theta^{\perp}} K(d,e_i^*,\theta) \cdot x^d\bigg)^{\<e_i,n\>},\]
where $(e_i^*)_{i\in V(Q)}\subset M$  denotes the dual basis to $(e_i)_{i\in V(Q)}\subset N$.
\end{thm}

In the case when $Q$ is acyclic (and hence $W=0$) we can use Theorem \ref{same} below to identify the stability scattering diagram with a purely combinatorial object called the cluster scattering diagram.
Theorem \ref{aut} then essentially coincides with   Reineke's result  \cite[Theorem 2.1]{reineke} as stated in \cite[Proposition 8.28]{ghkk}.

\subsection{Theta functions}

The theta functions defined by a scattering diagram are of crucial importance in the applications  to cluster varieties described in \cite{ghkk}. 
They are defined abstractly in terms of the wall-crossing automorphisms of the scattering diagram, but  can also be interpreted in terms  of counts of combinatorial objects called broken lines.

In general, theta functions are indexed by a lattice element $m\in M$. In this paper we only consider the functions $\vartheta^m(-)$ for elements lying in the positive cone
\[M^+=\{m\in M:m(n)>0\text{ for all }n\in N^+\}.\] 
To define these functions, consider the commutative Poisson algebra 
 \[B=\C[N^\oplus]\tensor_\C \C[M]=\C[x_1,\cdots,x_n][z_1^{\pm 1}, \cdots, z_n^{\pm 1}],\]
\[\{x^{n_1},x^{n_2}\}=\<n_1,n_2\>\cdot x^{n_1+n_2}, \quad \{x^n,z^m\}=m(n)\cdot x^{n}  z^m, \quad \{z^{m_1},z^{m_2}\}=0.\]
The Lie algebra $\g=\C[N]$ is a Lie subalgebra of $B$ so acts on it by derivations. Completing everything with respect to the $N^\oplus$-grading and exponentiating, we get an action of the group $\Hat{G}$ by algebra automorphisms of
\[\Hat{B}=\C[[x_1,\cdots,x_n]][z_1^{\pm 1},\cdots z_n^{\pm 1}].\]

Given a consistent scattering diagram $\DD$ as above we can  define for each $m\in M^+$  a theta function
\[\vartheta^m\colon M_\R\setminus \supp(\DD) \to \Hat{B}, \qquad \vartheta^m(\theta)=\Phi_\DD(\gamma)(z^m),\]
where $\gamma\colon [0,1]\to M_\R$ is any sufficiently general path from an arbitrary point $\theta_+\in M^+_\R$ to the point $\theta$.
The following result is closely related to Nagao's proof of the Caldero-Chapoton formula, although we do not explain the precise link here. More on this will appear in Man-Wai Cheung's thesis \cite{mandy}. 

\begin{thm}
\label{tor}
Let $(Q,W)$ be a 2-acyclic quiver with finite potential and $\DD$ the corresponding stability scattering diagram. Then for $m\in M^+$ and $\theta\in M_\R\setminus\supp(\D)$ there is an identity
\[\vartheta^{m}(\theta)=z^{m}\cdot \sum_{d\in N^\oplus } K(d,m,\theta) \cdot x^d. \]
\end{thm}

We leave for future research the problem of finding a similar moduli-theoretic description of the theta function $\vartheta^m(-)$ for general $m\in M$.
%

\subsection{Cluster scattering diagram}

 Let $Q$ be a 2-acyclic quiver  and consider  the Lie algebra $\g=\C[N^+]$ as above. Assume that the form $\<-,-\>$ is non-degenerate. Let $\DD$ be an arbitrary scattering diagram taking values in $\g$. Any wall $\dd$ of $\DD$ is contained in a hyperplace $n^\perp$ for a unique primitive element $n\in N^+$. We say that $\dd$ is incoming  if it contains the vector \[\theta_n=\<-,n\>\in M.\] Kontsevich and Soibelman proved that a consistent scattering diagram taking values in $\g$ is uniquely specified up to equivalence by its set of incoming walls and their associated wall-crossing automorphisms.

Consider the consistent scattering diagram in $\g$ whose only incoming walls are the hyperplanes $\dd_i=e_i^\perp$ with associated elements
\[\Phi_\DD(\dd_i)=\exp\bigg(\sum_{n\geq 1}\frac{x^{ne_i}}{n^2}\bigg) \in \Hat{G}.\]
It is this scattering diagram which plays a key role in \cite{ghkk}. We call it the  the {cluster scattering diagram} of  $Q$,
 since the adjoint action of $\Phi_\DD(\dd_i)$   is  given by the cluster transformation\[x^n\mapsto x^n\cdot (1+x^{e_i})^{\<e_i,n\>}.\]
 

It seems an interesting question to determine for which quivers the cluster scattering diagram can be realised as a stability scattering diagram for some appropriate choice of potential $W$. One result we have along these lines is

 \begin{thm}
\label{same}
If $Q$ is acyclic (and hence $W=0$) then the stability scattering diagram for $(Q,W)$ is equivalent to the   cluster scattering diagram associated to $Q$.
\end{thm}


\subsection*{Acknowledgements.} I am very grateful  to Arend Bayer, Mark Gross, Paul Hacking, Joe Karmazyn, Sean Keel and Alastair King for useful conversations on the subject of this paper, and to Fan Qin for pointing out an error in an earlier version.


 \section{Scattering diagrams}
\label{scatter}
 
 In this section we introduce basic definitions concerning scattering diagrams. We have to use a slightly different framework to that of \cite{ghkk} because the assumption that the group associated to a wall is  abelian  need not hold in the  general context in which we shall be working. Unproved assertions concerning convex rational  polyhedral cones  can be found in \cite[Section 1.2]{fulton}.

\subsection{}
\label{basic}
Let $N\isom \Z^{\oplus n}$ be a free abelian group of finite rank. Set \[M=\Hom_\Z(N,\Z), \quad M_\R=M\tensor_\Z \R.\] Fix a basis $(e_1,\cdots, e_n)$ for $N$ and set
\[N^\oplus=\big\{\sum_{i=1}^n \lambda_i e_i : \lambda_i\in \Z_{\geq 0}\big\}, \quad N^+=N^{\oplus}\setminus\{0\}.\]
Then $N^+$ is closed under addition, and $N^\oplus$ is a monoid.
We also consider the dual cone
\[M^+_\R=\{\theta\in M_\R:\theta(n)>0\text{ for all }n\in N^+\},\]
 and introduce notation
 \[M^+=M^+_\R\cap M, \quad M^\oplus =M^+\cup\{0\}, \quad M^-_\R:= -M_\R^+.\]

\subsection{}
We fix an $N^+$-graded Lie algebra
 \begin{equation}
\label{decomp}\g=\bigoplus_{n\in N^+} \g_n, \qquad [\g_{n_1},\g_{n_2}]\subset \g_{n_1+n_2}.\end{equation}
We shall often  consider the case when the Lie algebra $\g$ is nilpotent. There is then a corresponding algebraic group $G$  
with a bijective map
\begin{equation}
\label{exp}\exp\colon \g\to G.\end{equation}
If we use this map to identify $G$ with $\g$ then the  product  on $G$ is given by the Baker-Campbell-Hausdorff formula.
 
 \subsection{}
 By a \emph{cone}  in $M_\R$ we shall always mean a convex, rational, polyhedral cone, i.e. a subset of the form
 \[\sigma=\big\{\sum_{i=1}^p \lambda_i m_i:\lambda_i\in \R_{\geq 0}\big\}, \quad m_1,\cdots, m_p\in M.\]
 Any such cone has a dual description as an intersection of half-spaces:
 \[\sigma=\big\{\theta\in M_\R:\theta(n_i)\geq 0\text{ for } 1\leq i\leq q\big\}, \quad n_1,\cdots, n_q\in N.\]
 The \emph{codimension} of a cone is the codimension of the subspace of $M_\R$ it spans. We refer to codimension 1 cones as \emph{walls} and denote them by the symbol $\dd$.
 
 A \emph{face} of a cone $\sigma$ is a subset of the form \[\sigma\cap n^\perp=\{\theta\in \sigma:\theta(n)=0\},\]
where $n\in N$ satisfies $\theta(n)\geq 0$ for all $\theta\in \sigma$.
 Any face of a cone is itself a cone. Any intersection of faces of a given cone is also a face.
 
 \subsection{}
 A \emph{cone complex} in $M_\R$ is a finite collection $\S=\{\sigma_i\colon i\in I\}$ of  cones, such that
 \begin{itemize}
 \item[(a)] any face of a cone in $\S$ is also a cone in $\S$,\smallskip
 
 \item[(b)] the intersection of any two cones in $\S$  is a face of each.
 \end{itemize}
 The \emph{support} of a cone complex is the closed subset
 \[\supp(\S)=\bigcup_{\sigma\in \S} \sigma\subset M_\R.\]
  
\subsection{}
 \label{easy}\emph{Example}.
 Fix a finite subset $P\subset N^+$ and consider partitions \[P=P_+\sqcup P_0\sqcup P_-\] into disjoint subsets, with $P_0$ non-empty. To each such partition there is a cone
 \[\sigma(P_+,P_0,P_-)=\big\{\theta\in M_\R: \theta(n)=0\text{ for }n\in P_0\text{ and }\pm \theta(n)\geq 0\text{ for }n\in P_\pm\big\}\subset M_\R.\]
 The set of cones obtained from all such partitions of $P$ is a cone complex $\S(P)$ in $M_\R$.

 \subsection{}
 If $\sigma\subset M_\R$ is a  cone then we define a Lie subalgebra
\[\g(\sigma)=\bigoplus_{n\in  N^+\cap\sigma^{\perp}} \g_n\subset \g,\]
where $\sigma^\perp\subset N$ is the set of elements  orthogonal to all elements of $\sigma$.
Of course $\g(\sigma)$  depends only  on the subspace of $M_\R$ spanned by $\sigma$. 

\begin{defn}A \emph{$\g$-complex} $\DD=(\S,\phi)$ is a  cone complex $\S$  in $M_\R$   equipped with a choice $\phi(\dd)\in \g(\dd)$
for each wall $\dd\in \S$.
\end{defn}

The \emph{essential support} of a $\g$-complex $\DD=(\S,\phi)$ is the subset
\[\esupp(\DD)=\bigcup_{\dd\in\S:\phi(\dd)\neq 0} \dd\subset \supp(\S).\]

\subsection{}
Let $\DD=(\S,\phi)$ be a $\g$-complex.
We say that a smooth path $\gamma\colon [0,1]\to M_\R$ is \emph{$\DD$-generic} if
\begin{itemize}
\item[(a)] the endpoints $\gamma(0)$ and $\gamma(1)$ do not lie in the essential support of $\DD$,\smallskip

\item[(b)] $\gamma$ does not meet any cones of $\S$ of codimension $>1$,\smallskip

\item[(c)] all intersections of $\gamma$ with walls of $\S$ are transversal.
\end{itemize}
It follows  that there is a finite set of points $0<t_1<\cdots< t_k<1$ for which $\gamma(t_i)$ lies in the essential support of $\DD$, and at each of these points $t_i$ there is a unique wall $\dd_i\in \S$ such that $\gamma(t_i)\in \dd_i$.

\subsection{}Assume that the Lie algebra $\g$ is nilpotent. 
 Given a $\g$-complex $\DD=(\S,\phi)$ and a wall $\dd\in \S$ we can then define the element
\[\Phi_\DD(\dd)=\exp \phi(\dd)\in G.\]
Given any $\DD$-generic path $\gamma\colon [0,1]\to M_\R$  we can form the product
\[\Phi_\DD(\gamma)= \Phi_\DD(\dd_k)^{\epsilon_k}\cdot \Phi_\DD(\dd_{k-1})^{\epsilon_{k-1}}\cdots \Phi_\DD(\dd_2)^{\epsilon_2}\cdot \Phi_\DD(\dd_1)^{\epsilon_1}\in G,\]
where $\epsilon_i\in \{\pm 1\}$ is the negative of the sign of the derivative of  $\gamma(t)(n)$ at  $t=t_i$.

\begin{defn}
\begin{itemize}
\item[(a)]
We call two $\g$-complexes $\DD_1$ and $\DD_2$  \emph{equivalent} if, for any path $\gamma\colon [0,1]\to M_\R$ which is both $\DD_1$-generic and $\DD_2$-generic,  we have \[\Phi_{\DD_1}(\gamma)=\Phi_{\DD_2}(\gamma).\]

\item[(b)]
A $\g$-complex $\DD$  is called \emph{consistent}  if for any two $\DD$-generic paths $\gamma_i$ with the same endpoints we have \[\Phi_{\DD}(\gamma_1)=\Phi_{\DD}(\gamma_2).\]
\end{itemize}
\end{defn}



%

\subsection{}
\label{com}
Return now to the case of an arbitrary $N^+$-graded Lie algebra $\g$. For each $k\geq 0$  there is an ideal
\[\g^{> k} = \bigoplus_{\delta(n)> k} \g_n,\]
and a  nilpotent $N^+$-graded Lie algebra $\g_{\leq k}=\g/\g^{> k}$.  We denote the corresponding unipotent algebraic group by $G_{\leq k}$. For $i<j$ there are canonical homomorphisms
\begin{equation}
\label{direct}\pi^{ji}\colon \g_{\leq j}\to \g_{\leq i}, \qquad \pi^{ji}\colon G_{\leq j}\to G_{\leq i}.\end{equation}
We also consider the pro-nilpotent Lie algebra and the corresponding pro-unipotent algebraic group
 \[\Hat{\g}=\lim_{\longleftarrow} \;\g_{\leq k}, \qquad \Hat{G}=\lim_{\longleftarrow}\;G_{\leq k} .\]
Taking the limits of the maps \eqref{exp} gives a bijection $\exp\colon \Hat{\g}\to \Hat{G}$.

\subsection{}
In general, if $f\colon \g\to \h$ is a homomorphism of $N^+$-graded Lie algebras and $\DD=(\S,\phi)$ is a $\g$-complex, then there is an induced $\h$-complex \[f_*(\DD)=(\S,f\circ \phi).\]
Applying this to the maps \eqref{direct} we can make the following definition.

\begin{defn}
A \emph{$\Hat{\g}$-complex} $\DD$ is a sequence of $\g_{\leq k}$-complexes $\D_k$ for $k\geq 1$ such that for any $i<j$ the $\g_{\leq i}$ complexes  $\pi^{ji}_*(\D_j)$ and $\D_i$ are equivalent.
\end{defn}

Let $\DD=(\DD_k)_{k\geq 1}$ be a $\Hat{\g}$-complex.  We  set
\[\supp(\DD)=\bigcup_{k\geq 1} \supp(\DD_k),\quad \esupp(\DD)=\bigcup_{k\geq 1} \esupp(\DD_k).\]
 We say that  $\DD$ is  \emph{consistent} if each  $\g_{\leq k}$-complex $\D_k$ is consistent.  We say that $\DD$ is \emph{equivalent} to  another $\Hat{\g}$-complex $\DD'=(\DD'_k)_{k\geq 1}$ if for all $k\geq 1$ the $\g_{\leq k}$-complexes $\DD_k$ and $\DD'_k$ are equivalent.

\subsection{}Scattering diagrams are defined in \cite{ghkk}   under the assumption that for any $n\in N^+$ the sub-Lie algebra \[\g_{\<n\>}=
\bigoplus_{k\geq 0}\, \g_{kn}\subset \g\] is  abelian. 
In the next section we shall prove that when this assumption holds there is a natural bijection between equivalence classes of scattering diagrams in the sense of \cite{ghkk} and equivalence classes of $\Hat{\g}$-complexes as defined above. In the rest of the paper we shall therefore use the terms $\Hat{\g}$-complex and scattering diagram interchangeably.


\section{Reconstruction result}
\label{reconstruct}
In this section we reproduce Kontsevich and Soibelman's proof that scattering diagrams in $\g$ up to equivalence correspond to elements of the group $G$. 
Although our framework is slightly different to that of \cite{ks2} or \cite{ghkk}, there is nothing  new here, and all non-trivial statements (and their proofs) are due 
to Kontsevich and Soibelman. 

\subsection{}\label{threeone}
Continuing with the notation of the last section, let us suppose for the moment that the Lie algebra $\g$ is nilpotent. Suppose that $\DD=(\S,\phi)$ is a consistent $\g$-complex. If  $\theta_1,\theta_2\in M_\R$ lie outside the essential support of  $\DD$ then there is a  well-defined element
\[\Phi_\DD(\theta_1,\theta_2)\in G\]
obtained as $\Phi_\DD(\gamma)$ for any path $\gamma$ from $\theta_1$ to $\theta_2$.
Note that we always have
\[\esupp(\DD)\cap M_\R^\pm =\emptyset.\]
 Indeed, if $\dd\in \S$ is a wall then  $\phi(\dd)\neq 0$ implies that $\dd\subset n^\perp$ for some $n\in N^+$, but then $ M_\R^\pm\cap\dd=\emptyset$. Thus any $\g$-complex $\DD$ has a well-defined element \[\Phi_{\DD}=\Phi_\DD(\theta_+,\theta_-)\in G\] obtained as $\Phi_\DD(\gamma)$ for  a $\DD$-generic path $\gamma$ from a point $\theta_+\in M^+_\R$ to a point $\theta_-\in M^-_\R$.

\subsection{}
\label{han}
Given $\theta\in M_\R$ we can define Lie subalgebras of $\g$
\[\g_\pm(\theta)=\bigoplus_{n\in N^+:\pm \theta(n)>0}\g_n\quad \text{ and }\quad \g_0(\theta)=\bigoplus_{n\in N^+:\theta(n)=0} \g_n.\]
There is then a  decomposition
\begin{equation}
\label{de}\g=\g_-(\theta)\oplus \g_0(\theta)\oplus \g_+(\theta).\end{equation}
In the case that $\g$ is nilpotent we can consider the corresponding unipotent group $G$, and subgroups \[G_\star(\theta):=\exp\g_\star(\theta)\subset G, \quad \star\in\{+,-,0\}.\]
 The decomposition \eqref{de} implies that every element $g\in G$ has a unique decomposition \begin{equation}
\label{def}g=g_+\cdot g_0\cdot g_- \text{ with } g_\pm\in G_\pm(\theta)\text{ and }g_0\in G_0(\theta).\end{equation}
This defines projection maps (not group homomorphisms) $\Pi_\star^\theta\,\colon G\to G_\star(\theta)$.

\subsection{}
The following result shows that we can reconstruct all the the wall-crossing automorphisms of a  $\g$-complex  from the single element $\Phi_\DD$.

\begin{lemma}
\label{lem}
Assume that $\g$ is nilpotent and let $\DD=(\S,\phi)$ be a $\g$-complex. Then for any wall $\dd\in \S$ and any $\theta$ in the relative interior of $\dd$ we have the relation
\[\Phi_\DD(\dd)=\Pi_0^\theta\,(\Phi_\DD).\]
\end{lemma}

\begin{proof}
We can find an $\S$-generic straight-line path $\gamma\colon [0,1]\to M_\R$  connecting $M^+_\R$ to $M^-_\R$ and intersecting $\dd$. Let \[0<t_1<\cdots<t_k<1\] be the points for which $\theta_i=\gamma(t_i)$ lies on some wall $\dd_i$. In particular  $\dd=\dd_p$ for some unique $1\leq p\leq k$.  We set $\theta=\theta_p\in \dd$. The definition of the path-ordered product gives
\[\Phi_\DD=\Phi_\DD(\dd_k)\cdots \Phi_\DD(\dd_p)\cdots \Phi_\DD(\dd_1).\]
For each  wall $\dd_i$ there is a unique primitive element $n_i\in N^+$   such that $\dd_i\subset n_i^\perp$.   For $i>p$ the value of $\theta(n_i)$ lies between $\gamma(0)(n_i)>0$ and $\theta_i(n_i)=0$,  and hence  $\g(\dd_i)\subset \g_+(\theta)$. Similarly when $i<p$ we have $\g(\dd_i)\subset \g_-(\theta)$. Thus we conclude that
\[\Phi_\DD(\dd_i)\in \left\{\begin{array}{lll} \vspace{.15em}

G_+(\theta) &\text{ if $i>p$} \\ \vspace{.15em}

G_0(\theta) &\text{ if $i=p$} \\ 

G_-(\theta) &\text{ if $i<p$.}
\end{array} \right.\]The claim then follows from the uniqueness of the decomposition \eqref{def}.
\end{proof}

Note that it follows that the equivalence class of a $\g$-complex $\DD$ is  determined by the element $\Phi_\DD\in G$.

\subsection{}
Let us assume that there is a finite subset $P\subset N^+$ such that
\begin{equation}
\label{decomp2}\g=\bigoplus_{n\in P} \g_n.\end{equation}
 This implies in particular that $\g$ is nilpotent.
 Consider the corresponding cone complex $\S=\S(P)$ described in Section \ref{easy}. 
If $\dd\in \S$ is a wall then for any point $\theta$ in the relative interior of $\dd$   the decomposition \eqref{de} is constant, and moreover satisfies $\g_0(\theta)=\g(\dd)$. Given an element $g\in G$ we  can therefore define a $\g$-complex $\DD(g)=(\S,\phi)$ by taking
\[\Phi_{\DD(g)}(\dd)=\Pi_0^\theta\,(g),\]
where  $\theta\in \dd$ is any point in the relative interior of $\dd$.

\begin{lemma}
\label{surj}
The $\g$-complex $\DD(g)$ is consistent and satisfies $\Phi_{\DD(g)}=g$.
\end{lemma}

\begin{proof}The claim follows immediately from the statement that if  $\gamma\colon [0,1]\to M_\R$ is  any $\DD(g)$-generic path from $\theta_1$ to $\theta_2$ then
\[\Phi_{\DD(g)}(\gamma)=\Pi_+^{\theta_2}(g)^{-1}\cdot \Pi_+^{\theta_1}(g)=\Pi_-^{\theta_2}(g)\cdot \Pi_-^{\theta_1}(g)^{-1}.\]
To prove this note that  it is enough to check it for a  path crossing a single wall $\dd$. On the wall we have a decomposition $g=g_+\cdot g_0\cdot g_-$ as in \eqref{def} with \[g_0=\Phi_{\DD(g)}(\dd)=\Phi_{\DD(g)}(\gamma).\]
On the two sides of the wall $g_0$ becomes an element of either $G_+(\theta)$ or $G_-(\theta)$. Thus the decompositions \eqref{def} on the two sides of the wall have just two terms, and are $(g_+\cdot g_0) \cdot g_-$ and $g_+\cdot (g_0\cdot g_-)$ respectively. Comparing these gives the result.
\end{proof}

\subsection{}
The results of the last two sections together give 

\begin{prop}
\label{ks}
Suppose that  $\g$ has a finite decomposition \eqref{decomp2}. Then the map $\DD\mapsto \Phi_\DD$ defines  a bijection between equivalence classes of consistent $\g$-complexes and elements of the group $G$. \qed
\end{prop}

Consider now  a general $N^+$-graded Lie algebra $\g$, and let $\DD=(\DD_k)_{k\geq 1}$ be a $\Hat{\g}$-complex. Given two points  $\theta_1,\theta_2\in M_\R$ lying outside the essential support of $\DD$, the associated elements $\Phi_{\D_k}(\theta_1,\theta_2)\in G_{\leq k}$ are easily checked to be compatible with the group homomorphisms $\pi_{i,j}$. Taking the limit therefore gives an associated element
\[\Phi_{\DD}(\theta_1,\theta_2)\in \Hat{G}.\]
In particular, taking $\theta_1\in M_\R^+$ and $\theta_2\in M_\R^-$ we get a well-defined element $\Phi_{\DD}\in \Hat{G}$.
 Since an equivalence class of $\Hat{\g}$-complexes is nothing but a compatible sequence of equivalence classes of $G_{\leq k}$-complexes, we immediately have

\begin{prop}
\label{ks2}
The map $\DD\mapsto \Phi_\DD$ defines  a bijection between equivalence classes of consistent $\Hat{\g}$-complexes and elements of the group $\Hat{G}$.
\qed
\end{prop}

It follows that $\Hat{\g}$-complexes  up to equivalence coincide with scattering diagrams up to equivalence as defined in \cite{ghkk}.


\section{Representations of quivers}

This section contains basic definitions and results concerning representations of quivers with relations. These are all of course well-known but it is not so easy to find good references for the material on moduli stacks.

\subsection{}
\label{quivnot}
A quiver is a finite oriented graph specified by sets  $(V(Q),A(Q))$ of vertices and arrows respectively, and source and target maps $s,t\colon A(Q)\to V(Q)$. We write $\C Q$ for the path algebra of $Q$, and $\C Q_{\geq k}\subset \CQ$ for the subspace spanned by  paths of length $\geq k$. For our purposes a quiver with relations is a pair $(Q,I)$, where $Q$ is a quiver, and $I\subset \C Q_{\geq 2}$ is a two-sided ideal spanned by linear combinations of paths of length at least 2. We denote by  \[\A=\rep(Q,I)=\mod \C Q/I\] the abelian category   of finite-dimensional representations of the pair $(Q,I)$, or equivalently  finite-dimensional left modules for the quotient algebra $\C Q/I$.
There is a group homomorphism
\[d\colon K_0(\A)\to \Z^{V(Q)}\]
sending a representation to its dimension vector. 
We set \[N=\Z^{V(Q)}, \quad M=\Hom_\Z(N,\Z).\]
We denote by $(e_i)_{i\in V(Q)}\subset N$ the canonical basis indexed  by the vertices of $Q$.
The corresponding  positive cone
$N^+\subset N$ consists  of dimension vectors of nonzero objects of $\A$. 
We define $\delta\in M$ to be the element for which $\delta(E)=\dim_\C(E)$ for any representation $E\in \A$.
We use the notation $M^+$,  $M_\R$, $N^\oplus$ etc as in Section \ref{basic}.

\subsection{}
\label{blib}
There is an algebraic stack  $\M$ parameterizing all objects of the category $\A$. It is defined formally as a fibered category over the category of schemes as follows. The objects of  $\M$ over a scheme $S$ are pairs $(\E,\rho)$ where $\E$ is a locally-free $\O_S$-module of finite rank, and  \[\rho\colon \C Q/I\to \End_S(\E)\]
is an algebra homomorphism.
If $(\E',\rho')$ is another object of $\M$ lying over a scheme $S'$, then
a morphism  \[(\E',\rho')\to (\E,\rho)\]
in $\M$  lying over a morphism of schemes $f\colon S'\to S$ is an isomorphism of $\O_S$-modules $\theta\colon f^*(\E)\to \E'$ which intertwines the maps $\rho$ and $\rho'$. Here we have taken the usual step of choosing, for each morphism of schemes, a pullback of every coherent sheaf on its target. The intertwining condition is that for any $a\in \CQ/I$ there is a commuting diagram
\[\begin{CD} f^*(\E) &@>f^*(\rho(a))>> &f^*(\E) \\ @V\theta VV && @V\theta VV \\ \E' & @>\rho'(a)>>&\E'\end{CD}\]
The stack axioms follow easily from the corresponding statements for the stack of locally-free sheaves.
We prove that $\M$ is algebraic in the next subsection.

\subsection{}The following statement is extremely well-known, but one does not often find a treatment in the language of stacks, so we briefly indicate a proof.

\begin{lemma}
The stack $\M$ splits as a disjoint union
\begin{equation}
\label{disjoint}\M=\big.\coprod_{d\in N^\oplus} \M_d\end{equation}
 of open and closed substacks $\M_d$ parameterizing representations of  a fixed dimension vector. Each of these substacks can be presented as a  quotient of an affine variety by an affine algebraic group.
 \end{lemma}
 
 \begin{proof}
The usual equivalence of categories between modules for the algebra $\CQ/I$ and representations of  $(Q,I)$  extends to show that the groupoid $\M(S)$ can be equivalently described in terms of representations of $Q$ in locally-free $\O_S$-modules. Thus an $S$-valued point  of $\M$ can be taken to consist of locally-free $\O_S$-modules $\V_i$ for each vertex $i\in V(Q)$ and morphisms $\rho(a)\colon \V_{s(a)}\to \V_{t(a)}$ for each arrow $a\in A(Q)$, such that the   relations in $I$ are satisfied.  The first statement then follows from the fact that the rank of each sheaf $\V_i$ is locally constant on $S$.
 
 Let us now fix a dimension vector $d\in N^{\oplus}$ and define the algebraic group
 \[\GL(d)=\prod_{i\in V(Q)} \GL(d_i).\]
 Consider an $S$-valued point of the substack $\M_d$.  Taking the frame bundles associated to the vector bundles $\V_i$ defines a principal $\GL(d)$ bundle $\pi\colon P(S)\to S$. Pulling back the representation $(\V_i,\rho_a)$ to $P(S)$, the bundles $\V_i$ become canonically trivialised and the representation corresponds to  a map from $P(S)$ to the closed subvariety
\[\Rep(d)\subset \prod_{a\in A(Q)} \Hom_{\C}(\C^{d_{s(a)}},\C^{d_{t(a)}})\]
cut out  by the given relations. In this way one sees that 
\[\M_d\isom [\Rep(d)/\GL(d)],\]
where $\GL(d)$ acts on $\Rep(d)$ by gauge transformations\[(g_i)_{i\in V(Q)} \cdot (\phi_a)_{a\in A(Q)}=(g_{t(a)}^{-1} \cdot \phi_a \cdot g_{s(a)})_{a\in A(Q)}.\]
We leave the  reader to fill in the details of this argument.
\end{proof}

It follows in particular from this  that $\M$ is an Artin stack, locally of finite-type over $\C$ and with affine diagonal.

\subsection{}
We shall also need the stack $\M^{(2)}$ of short exact sequences in $\A$. The objects of $\M^{(2)}$ over a scheme $S$ consist of  objects $(\E_i,\rho_i)$ of $\M(S)$ for $i=1,2,3$, together with morphisms $\alpha,\beta$ of $\O_S$-modules which intertwine the maps $\rho_i$, and  form  a short exact sequence
\begin{equation}
\label{finger}0\lra \E_1\lRa{\alpha} \E_2\lRa{\beta} \E_3\lra 0\end{equation}
of $\O_S$-modules.
Given another such object defined by objects $(\E'_i,\rho'_i)$ of $\M(S')$ and morphisms $\alpha',\beta'$, a morphism between them lying over $f\colon S'\to S$ is given by a commuting diagram of morphisms of $\O_{S'}$-modules 
\[\begin{CD} 0 @>>> &f^*(\E_1) @>f^*(\alpha)>> &f^*(\E_2) @>f^*(\beta)>> &f^*(\E_3) @>>> 0\\
&&&@V\theta_1VV &@V\theta_2VV &@V\theta_3VV \\
0 @>>> & \E'_1 @>\alpha'>> &\E'_2 @>\beta'>> &\E'_3 @>>> 0\end{CD}\]
in which vertical arrows are isomorphisms intertwining the actions of $\rho_i,\rho'_i$ as before. The stack axioms follow easily from those for $\M$.

\subsection{}
There is a
 diagram of morphisms of stacks
\begin{equation}\label{ba}\begin{CD}
\M^{(2)} &@>b>> \M\\
 @V(a_1,a_2)VV \\
\M\times\M\end{CD}\end{equation}
where  
$a_1,a_2$ and $b$ send a short exact sequence \eqref{finger}
 to the objects $\E_1,\E_3$ and $\E_2$ respectively.

\begin{lemma}
\label{call}
\begin{itemize}
\item[(a)]  The morphism $(a_1,a_2)$ is of finite type.
\item[(b)] The morphism $b$ is representable and proper.
\end{itemize}
\end{lemma}

\begin{proof}
We prove (b) first. 
Suppose $f\colon S\to \M$ is a morphism with $S$ a scheme. It corresponds to a pair $(\E,\rho)$ as before.  The objects of the fibre product stack $S\times_\M \M^{(2)}$ over a scheme $f\colon T\to S$ are short exact sequences of locally-free $\O_T$-modules
\[0\lra \E_1\lra f^*(\E)\lra \E_3\lra 0\]
such that $f^*(\rho(a))(\E_1)\subset \E_1$ for all $a\in \CQ/I$. This is represented by a closed subset of the relative Grassmannian  of the sheaf $\E$ over $S$. 
To prove (a) note that \[(a_1,a_2)^{-1}(\M_{d_1}\times\M_{d_2})=b^{-1}(\M_{d_1+d_2}).\] This stack is of finite type since both the stack $\M_{d_1+d_2}$ and the morphism  $b$ are. It follows that the morphism $(a_1,a_2)$ is also of finite type. 
\end{proof}

Lemma \ref{call} (b) implies that $\M^{(2)}$ is an algebraic stack, since pulling back an atlas for $\M$ gives an atlas for $\M^{(2)}$. Note also that the morphism $(a_1,a_2)$ is  not representable. The fibre  over a
point of $\M \times \M$ corresponding to a pair of  representations
$(E_1,E_3)$ is the quotient stack
\[[\Ext^1(E_3,E_1)/\Hom(E_3,E_1)],\]
with the action of the vector space $\Hom_\A(E_3,E_1)$ being the
trivial one. See  \cite[Proposition 6.2]{bridgeland2} for a proof of this fact (which we shall not explicitly use in what follows).


 \section{Motivic Hall algebras}
\label{hall}

 To construct a suitable Lie algebra for our scattering diagram we consider motivic Hall algebras as introduced by Joyce \cite{joyce1}. The construction we need is reviewed in detail in \cite{bridgeland2} in the case of categories of coherent sheaves on smooth projective varieties, and only very minor modifications are required for  the case of categories of representations of quivers with relations. Motivic Hall algebras of quiver representations  were also used by Nagao \cite{nagao}.

\subsection{}
Let $\M$ be an Artin stack, locally of finite type over $\C$ and with affine stabilizers. There
is a 2-category of algebraic stacks over $\M$.
 Let $\operatorname{St/\M}$ denote the full subcategory  consisting of objects
\begin{equation}
\label{obvious2}f \colon X \to \M\end{equation} for which $X$ is of
finite type over $\C$ and has affine stabilizers. 

\begin{defn}
\label{rel} Let $\kst{\M}$ be the free abelian group with basis given by
isomorphism classes of objects  of
$\operatorname{St/\M}$,
  modulo the subgroup spanned by relations
  \begin{itemize}
\item[(a)] for every object \eqref{obvious2}, and every closed substack $Y\subset X$ with complementary open substack $U=X\setminus Y\subset \M$, a relation
\[[X\lRa{f}\M]=[Y\lRa{f|_Y} \M] + [U\lRa{f|_U} \M].\]

\item[(b)]for every object \eqref{obvious2}, and every pair of morphisms 
\[h_1\colon Y_1 \to X, \quad h_2\colon Y_2\to X\]
which are locally-trivial  fibrations in the Zariski topology
  with the same fibres,  a relation
\[[Y_1\lRa{g\circ h_1} \M]=[Y_2\lRa{g\circ h_2} \M].\]

\end{itemize}
\end{defn}

For (b) we recall that a morphism of stacks is said to be a locally-trivial fibration in the Zariski topology if it is representable, and if any pullback to a scheme is a locally-trivial fibration of schemes in  the Zariski topology. 

\subsection{}
\label{sunny}
We shall frequently use the following simple observation.
Suppose we have a  commutative diagram
\[ \xymatrix@C=.6em{
X_1\ar[dr]_{f_1}\ar[rr]^{g} && X_2\ar[dl]^{f_2}\\
&\M }
\]with $X_1,X_2$ algebraic stacks of finite type over $\C$ with  affine stabilizers as above. 
Suppose also that the morphism of stacks $g$ induces an equivalence at the level of $\C$-valued points. Then we have an equality  
\[[X_1\lRa{f_1} \M]=[X_2\lRa{f_2} \M],\]
in the group $K(\St/\M)$. See  \cite[Sections 2--3]{bridgeland2} for more details. The basic point is that given any such morphism $g$, one can stratify the stacks $X_i$ by locally-closed substacks so that $g$ becomes an isomorphism on each substack.

\subsection{}
The group $\kst{\M}$ has the structure of a $\kst{\C}$-module,
defined by setting
\[[X]\cdot [Y\lRa{f} \M]=[X\times Y\lRa{f\circ\pi_2} \M]\]
and extending linearly.
There is a unique  ring homomorphism
\[\Upsilon \colon \kst{\C}\lra \C(t)\]
which takes the class of a smooth projective variety over $\C$ (considered as a representable stack) to the Poincar{\'e} polynomial
\[\Upsilon([X])=\sum_{i=0}^{2d} \dim_\C H^i(X_{\an},\C)\cdot t^i\in \C[q].\]
Here $X_{\an}$ denotes $X$ considered as a compact complex manifold, and $H^i(X_{\an},\C)$ denotes  singular cohomology.
We can therefore consider 
\begin{equation}\label{vect}K_\Upsilon(\St/M)=K(\St/\M)\tensor_{K(\St/\C)} \C(t).\end{equation}
More concretely, this is the $\C(t)$ vector space with basis the  isomorphism classes of objects \eqref{obvious2} as above, modulo the relations (a)--(b) of Definition \ref{rel} and the extra relations
\[[X\times Y\lRa{f\circ\pi_2} \M]=\Upsilon([X])\cdot [Y\lRa{f} \M]\]
The resulting vector space is what Joyce  \cite[Section 4.3]{joyce1}  denotes $\operatorname{SF}(\M,\Upsilon,\C(t))$.

\subsection{}
The abelian group $K(\St/\M)$ becomes a ring  when equipped with the convolution product coming from the diagram \eqref{ba}. Explicitly this product is given by the rule
\[ [X_1\lRa{f_1}\M] * [X_2\lRa{f_2} \M] = [Z\lRa{b\circ h}\M], \]
where  $h$ is defined by the following Cartesian
square
\[\begin{CD}
Z & @>h>> &\M^{(2)} &@>b>> \M\\
 @VVV  &&@VV(a_1,a_2)V \\
X_1\times X_2 &@>f_1\times f_2>> &\M\times\M\end{CD}\]
We refer the reader to \cite[Section 4]{bridgeland2} for more details.
The convolution product is easily seen to be $K(\St/\C)$-linear and so also defines an algebra structure on the $\C(t)$-vector space $K_\Upsilon(\St/\M)$. We denote the resulting $\C(t)$-algebra by
$H(Q,I)$.


\subsection{}
\label{dully}
Let us consider the general situation of a graded   algebra
\[A=\bigoplus_{n\in N^\oplus} A_n, \qquad A_{n_1}\cdot A_{n_2}\subset A_{n_1+n_2}.\]
For each $k\geq 1$ we can consider the ideal $A_{>k}=\oplus_{\delta(n)>k} \, A_n,$, and the corresponding quotient $A_{\leq k}=A/A_{\geq k}$. Taking the limit of the obvious maps gives the algebra 
 \[\Hat{A}=\lim_{\longleftarrow} \;A_{\leq k} .\]
 The subspace $\g=A_{>0}\subset A$ is a Lie algebra under commutator bracket. The corresponding completion $\Hat{\g}=\Hat{A}_{>0}$ is a pro-nilpotent Lie algebra under commutator bracket. As usual, we denote the corresponding pro-unipotent group by $\Hat{G}$.

 \subsection{}
 \label{dull}
 Continuing the notation from Section \ref{dully}, we recall
the following  standard facts.

 \begin{itemize}
 \smallskip
 
\item[(i)]There is an embedding
\[\phi\colon \Hat{G}\hookrightarrow \Hat{A}\]
which sends an element $\exp(x)\in \Hat{G}$  to the element of $\Hat{A}$ obtained by taking the exponential of $x\in \Hat{\g}$ inside the algebra $\Hat{A}$. The Baker-Campbell-Hausdorff formula  ensures that\[\phi(g_1\cdot g_2)=\phi(g_1)\cdot \phi(g_2).\]
Thus we can identify $\Hat{G}$ with the subset
$\one+\Hat{A}_{>0} \subset \Hat{A}.$\smallskip

\item[(ii)] Exponentiating the adjoint action of $\Hat{\g}$ on $\Hat{A}$ gives an action of  $\Hat{G}$  by  algebra automorphisms on $\Hat{A}$. Under  the embedding $\phi$, this corresponds to conjugation:
\[
\exp(x)(a):=\exp \,[x,-] (a)=\phi(x)\cdot a\cdot \phi(x)^{-1}.\]
\end{itemize}
We now apply these general ideas to the  Hall algebra $H(Q,I)$.

\subsection{}

The decomposition \eqref{disjoint} induces a grading
\begin{equation}
\label{phd} H(Q,I)=\bigoplus_{d\in N^{\oplus}} H(Q,I)_d.\end{equation} 
where $H(Q,I)_d=K_\Upsilon(\St/\M_d)$. We define the Lie subalgebra \[\g_{\Hall}:=H(Q,I)_{>0}.\]%

Note that
$\M_0= \Spec(\C)$
is a point, and the identity of the algebra $H(Q,I)$ is represented by the symbol  \[\one=\big[\M_0\subset \M\big].\]

Applying the general statements of the last two subsections gives a completed algebra $\Hat{H}(Q,I)$ with a Lie subalgebra \[\Hat{\g}_{\Hall}:=\Hat{H}(Q,I)_{>0}\subset \Hat{H}(Q,I),\]
 and a corresponding pro-unipotent group $\Hat{G}_{\Hall}$, with an identification
 \begin{equation}
 \label{sub}\Hat{G}_{\Hall} \isom \one+\Hat{H}(Q,I)_{>0}\subset \Hat{H}(Q,I).\end{equation}
 
\subsection{}
\label{spelement}
Suppose given a morphism of stacks
\[f\colon X \to \M\]  such that $X$ has affine stabilizers but is not necessarily of finite-type over $\C$. Suppose instead  that for any $d\in N^\oplus$, the closed and open substack $X_d:=f^{-1}(\M_d)$ is of finite type. Then  the induced morphisms
\[ f_{\leq k}\colon \bigcup_{\delta(d)\leq k} X_d \lra \M\]
define compatible elements of the truncations $H(Q,I)_{\leq k}$, and so we obtain an element of the completed Hall algebra, which we abusively denote \[[X\lRa{f}\M]\in \Hat{H}(Q,I).\]
This element can be identified via \eqref{sub} with an element of the group $\Hat{G}_{\Hall}$ 
precisely if the map $f_0\colon X_0\to \M_0$ is an isomorphism.

\subsection{}
\label{puff}
Recall that $H(Q,I)$ is an algebra over the field $\C(t)$. Define a subalgebra
\[\C_\reg(t)=\C[t,t^{-1}][(1+t^2+\cdots +t^{2k})^{-1} :k\geq 1]\subset \C(t).\]
Thus we invert the Poincar{\'e} polynomials of the affine line $\mathbb{A}^1$ and all projective spaces $\mathbb{P}^k$. Let
\begin{equation}\label{guff}H_{\reg}(Q,I)\subset H(Q,I)\end{equation}
denote the $\C_\reg(t)$--submodule  generated by symbols \eqref{obvious2} with $X$ a variety. Note that by definition $H_{\reg}(Q,I)$ is a \emph{free}  $\C_\reg(t)$-module.

\begin{thm}
\label{co}
The subset \eqref{guff}   is closed under the convolution product on $H(Q,I)$ and is therefore a $\C_\reg(t)$-algebra. Moreover the quotient \[H_{\sc}(Q,I)=H_{\reg}(Q,I)\big/\,(t^2-1)H_{\reg}(Q,I)\] is commutative.
\end{thm}

\begin{proof}
This is proved exactly as in \cite[Theorem 5.1]{bridgeland2}. Note that the Poincar{\'e} polynomial of the affine line is $\Upsilon(\L)=t^2$. 
\end{proof}

This result has two closely related consequences:

\begin{itemize}
\item[(i)]The Poisson bracket on $H(Q,I)$ defined by
\begin{equation}
\label{brack}\{a,b\}=(t^2-1)^{-1}\cdot [a,b].\end{equation}
 induces a Poisson bracket on the subring $H_{\reg}(Q,I)$. 
\item[(ii)]The subspace
\begin{equation}
\label{yes}
\g_{\reg}:=(t^2-1)^{-1} \cdot H_{\reg}(Q,I)_{>0} \subset \g_{\Hall}= H(Q,I)_{>0}\end{equation}
is a Lie subalgebra. 

\end{itemize}

Note that $H_{\reg}(Q,I)_{>0}$ viewed as a Lie algebra via its Poisson bracket is isomorphic to the Lie algebra $\g_\reg$ via the  map $x\mapsto (t^2-1)^{-1}\cdot x$.

%

%

\subsection{}
Continuing with the notation of the last section, we can pass to completions and define a Lie algebra
\[\Hat{\g}_{\reg}:= (t^2-1)^{-1} \cdot \Hat{H}_{\reg}(Q,I)_{>0}\subset \Hat{H}(Q,I)_{>0}.\]
The corresponding pro-unipotent group 
$\Hat{G}_{\reg} \subset \Hat{G}_{\Hall}$ has an identification
\begin{equation}
\label{dastardly}
\Hat{G}_\reg=\exp\big((t^2-1)^{-1}\cdot \Hat{H}_\reg(Q,I)_{>0}\big)\subset \one+\Hat{H}(Q,I)_{>0}.\end{equation}

In Section \ref{load} we shall make use of the following deep theorem of Joyce.  

\begin{thm}[Joyce]
\label{biggy}
For any $\theta\in M_\R$ the element of $\Hat{G}\subset \Hat{H}(Q,I)$ defined by the inclusion of the open substack of $\theta$-semistable objects \[\M_{\ss}(\theta)\subset \M\]   corresponds via the identification \eqref{dastardly} to an element $1_\ss(\theta)\in \Hat{G}_\reg$.
\end{thm}

The reader can consult \cite[Section 6.3]{bridgeland3} for precise references for this result.

%


 \section{The Hall algebra scattering diagram}
 \label{hallscatter}

In this section we construct a canonical scattering diagram associated to the category of finite-dimensional representations of a quiver with relations. It takes values in the Lie algebra $\g_{\Hall}$ defined in the last section. The walls correspond to choices of weights $\theta\in M_\R$ for which there exist nonzero $\theta$-semistable representations.

\subsection{}
\label{sixa}
Let $(Q,I)$ be a quiver with relations as in Section \ref{quivnot} and take notation as defined there.  The following definition is due to King \cite{king}.\footnote{Note however that King takes the inequality of Definition \ref{bobob}(ii) to be $\theta(A)\geq 0$. In the present context our convention seems to be preferable since it avoids the introduction of an extra sign in Section \ref{alal}.}
\begin{defn}
\label{bobob}
Given  $\theta\in M_\R$,  an object $E\in \A$ is said to be \emph{$\theta$-semistable} if 
\begin{itemize}
\item[(i)] $\theta(E)=0$,
\item[(ii)]  every subobject $A\subset E$ satisfies $\theta(A)\leq 0$.
\end{itemize}
\end{defn}

For $k\geq 1$ we define  $\A_{\leq k}\subset \A$ to be the full subcategory consisting of representations of total dimension  $\leq k$.
We  define a subset
\[\W_{k}=\big\{\theta\in M_\R: \text{there exist $\theta$-semistable objects $0\neq E\in \A_{\leq k}$}\big\}\subset M_\R.\]
The following result implies in particular that $\W_{k}\subset M_\R$ is closed.
\begin{lemma}
\label{clang}
There is a cone complex $\S_{k}$ in $M_\R$ such that
\[\W_{k}=\supp(\S_{k}).\]
\end{lemma}

\begin{proof}
 First consider the finite subset
\[ P=\{n\in N^+\colon \delta(n)\leq k\}\subset N^+,\]
and the corresponding cone complex $\S(P)$ of Example \ref{easy}. Note that in the relative interior of each cone of $\S(P)$ the question of the semistability of any given object of $\A_{\leq k}$ is constant. Thus  we can define a subcomplex consisting of those cones which support nonzero semistable objects in $\A_{\leq k}$. The support of this subcomplex is precisely $\W_{k}$.
\end{proof}

 \subsection{}
 Given $\theta\in M_\R$ we define a \emph{stability function} $Z\colon K_0(\A)\to \C$ by the formula
 \[Z(E)=-\theta(E)+i \delta(E).\]
 Note that if $0\neq E\in\A$ then $Z(E)\in \mathcal{H}$ lies in the upper half-plane, and we can define the \emph{phase} of $E$ by
 \[\phi(E)=(1/\pi)\arg Z(E).\]
 We then define an object $0\neq E\in \A$ to be $Z$-semistable if every nonzero subobject $A\subset E$ satisfies $\phi(A)\leq \phi(E)$.  The usual argument shows that if $E_1,E_2\in \A$ are $Z$-semistable with phases $\phi_1,\phi_2$ respectively then
 \begin{equation}
 \label{no}\phi_1>\phi_2\implies \Hom_\A(E_1,E_2)=0.\end{equation}
 Note that a nonzero object $ E\in \A$ is $\theta$-semistable precisely if it is $Z$-semistable with phase $1/2$.
 
 It is an elementary fact that every nonzero object $E\in \A$ has a unique Harder-Narasimhan filtration
 \[0=E_0\subset E_1\subset \cdots \subset E_{m-1}\subset E_m=E,\]
 whose factors $F_i=E_i/E_{i-1}$ are $Z$-semistable with descending phase:
 \[\phi(F_1)> \phi(F_2)>\cdots >\phi(F_m).\]
 For any interval $I\subset (0,1)$ we define $\P(I)\subset \A$ to be the full additive subcategory consisting of objects all of whose Harder-Narasimhan factors $F_i$ have phases in $I$. In particular, if $\phi\in (0,1)$ then $\P(\phi)=\P(\{\phi\})$ denotes the subcategory of $Z$-semistable objects of phase $\phi$.

\subsection{}

Fix an element $\theta\in M_\R$ and consider the stability function $Z$ as  in the last section.
\begin{lemma}
\label{witch}
For any interval $I\subset (0,1)$ there is an open substack
$\M_I(\theta)\subset \M$ parameterising representations of $(Q,I)$  lying in the full subcategory $\P(I)\subset \A$.
\end{lemma}

\begin{proof}
It is enough to prove the same statement for the stack $\M_d$ parameterising representations of a fixed dimension vector $d\in N^\oplus$. There are then only finitely many elements $n\in N^\oplus$  occurring as possible dimension vectors of subobjects. Thus it is enough to know that for any $n\in N^\oplus$ there is an open substack of $\M_d$ parameterizing representations having no subobject of type $n$. 
Consider the diagram \eqref{ba} and the closed substack $a_1^{-1}(\M_{n})$. The result then follows from the properness of the morphism $b$. \end{proof}

Note that the zero object of $\A$ lies in the category $\P(I)$ for any interval $I$. Thus $\M_I(\theta)\cap \M_0=\M_0$.  As in Section \ref{spelement}, the substack of Lemma \ref{witch} defines an element \[1_{I}(\theta)=[\M_I(\theta)\subset \M] \in \Hat{G}_{\Hall}\subset \Hat{H}(Q,I).\]
A particular case of this  corresponding to the interval $I=\{\halft\}$ is
\[1_\ss(\theta)=[\M_{\ss}(\theta)\subset \M]\in \Hat{G}_{\Hall}.\]
In the case of the interval $I=(0,1)$ we use the notation
\[1_\A = 1_{(0,1)} = [\M\lRa{\id}\M]\in \Hat{G}_{\Hall}.\]
This should not be confused with the identity element $\one=[\M_0\subset \M]$.

\subsection{}
The existence and uniqueness of the Harder-Narasimhan filtration gives rise to an important identity in the Hall algebra $H(Q,I)$. It is this identity that is responsible for the link between stability conditions and scattering diagrams.

\begin{prop}
\label{inty}
Fix $\theta\in M_\R$ and take an interval $I\subset (0,1)$ which  is a disjoint union of intervals $I_1,I_2$. Suppose that $I_1>I_2$ in the obvious sense. Then there is an identity \[1_I(\theta)=1_{I_1}(\theta) * 1_{I_2}(\theta)\in \Hat{G}_\Hall.\]
\end{prop}

\begin{proof}
 The result follows immediately from the statement that for any interval $I\subset (0,1)$ there is an identity  \begin{equation}
 \label{edin}1_I(\theta)=\prod_{\phi\in I} 1_{\phi}(\theta)\in \Hat{H}(Q,I),\end{equation}
 where the product over phases is taken in descending order. Note that the right hand side of \eqref{edin} makes good sense because for a fixed $k\geq 1$ there are only  finitely many $\phi\in I$ for which there exists a nonzero semistable object in $\A_{\leq k}$ of phase $\phi$. Thus when projected to $H(Q,I)_{\leq k}$, all but finitely many factors are the identity.
 
To prove \eqref{edin} let us indeed project it to $H(Q,I)_{\leq k}$, where it takes the form
\begin{equation}
\label{su}1_I(\theta)= 1_{\phi_1}(\theta)*1_{\phi_2}(\theta)*\cdots 1_{\phi_r}(\theta),\end{equation}
for some finite sequence of possible phases $\phi_1>\phi_2>\cdots>\phi_r$. Recall that the $r$-fold product in the Hall algebra is given by convolution using the  stack $\M^{(r)}$ parameterizing objects of $\rep(Q,I)$ equipped with an $r$-step filtration (see \cite[Lemma 4.4]{bridgeland2}).
The right hand side of \eqref{su}  is  therefore represented by the open substack $\N\subset \M^{(r)}$
consisting of  filtered objects  which when pulled back to any $\C$-valued point give a filtration by semistable objects of the given phases $\phi_i$. Since any object in $\P(I)\cap\A_{\leq k}$ has a unique such filtration (the Harder-Narasimhan filtration extended to length $r$ by inserting zero factors), it follows that the obvious map of stacks \[\N\cap b^{-1}(\M_{\leq k})\to \M_{I}(\theta)\cap \M_{\leq k}\] forgetting the filtrations induces an equivalence on $\C$-valued points. The identity then follows from the remark of Section \ref{sunny}.
\end{proof}

\subsection{}
\label{bustfinger}

The following result gives the fundamental link between stability conditions and scattering diagrams.

 \begin{thm}
\label{hallsc}
There is a consistent scattering  diagram $\DD$ in $M_\R$ taking values in $\g_\Hall$  and with the following properties:\begin{itemize}
\item[(a)]
The support  consists of those elements $\theta\in M_\R$ for which there exist nonzero $\theta$-semistable objects in $\A$.\smallskip

\item[(b)]  The wall-crossing automorphism at a general point $\theta\in \dd\subset \supp(\DD)$ is 
\[\Phi_\DD(\dd)=1_{\ss}(\theta)\in \Hat{G}_\Hall.\]

\end{itemize}
\end{thm}

\begin{proof} 
Lemma \ref{inty} shows that for any $\theta\in (0,1)$ we have a relation in $\Hat{G}_\Hall$
\begin{equation}
\label{crux}1_\A=1_{(\halft,1)}(\theta) * 1_{\ss}(\theta) * 1_{(0,\halft)}(\theta).\end{equation}
Let us fix $k\geq 1$ and consider the nilpotent Lie algebra $\g=\g_{\Hall,\leq k}$ and the corresponding unipotent group $G=G_{\Hall,\leq k}$. 
It is easy to see that, using the notation of Section \ref{threeone}, for any phase $\phi\in (0,1)$ we have
\[\pi_{\leq k} (1_\ss(\phi))\in  \left\{\begin{array}{lll} \vspace{.2em}

G_+(\theta) &\text{ if $\phi>\halft$} \\ \vspace{.2em}

G_0(\theta) &\text{ if $\phi=\halft$} \\ 

G_-(\theta) &\text{ if $\phi<\halft$.}
\end{array} \right.\]\
Equation \eqref{crux} therefore gives identities
\begin{equation}
\label{readit}\Pi_+^\theta(1_\A)=1_{(\halft,1)}(\theta), \quad  \Pi_0^\theta(1_\A)=1_{\ss}(\theta), \quad \Pi_-^\theta(1_\A)= 1_{(0,\halft)}(\theta).\end{equation}
The general construction of Section \ref{reconstruct} then gives a consistent scattering diagram satisfying (b). Passing to a subcomplex  as in the proof of Lemma \ref{clang} we can ensure that it also satisfies (a).
\end{proof}

The scattering diagram of Theorem \ref{hallsc} is  is clearly unique up to equivalence. We call it the \emph{Hall algebra scattering diagram} of $(Q,I)$.
 By construction, the corresponding element  of 
 Proposition \ref{ks2} is 
\[\Phi_\DD=1_\A\in \Hat{G}_\Hall.\] 
Note that Joyce's result Theorem \ref{biggy} is precisely the statement that the stability scattering diagram takes values in the Lie subalgebra $\g_{\reg}\subset \g_{\Hall}$. 

 \subsection{}
 It is interesting and useful to relate our Hall algebra scattering diagram to a certain  torsion pair in $\A$ defined by an element $\theta\in M_\R$. 

\begin{lemma}
\label{strain}
For each $\theta\in M_\R$ there is a torsion pair $(\T(\theta),\F(\theta))\subset \A$ defined by setting
\vspace{-.7em}

\begin{gather*}\T(\theta)=\P(\halft,1)=\{E\in \A:\text{any quotient object }E\onto Q \text{ satisfies $\theta(Q)> 0$}\},\\
\F(\theta)=\P(0,\halft]=\{E\in \A:\text{any subobject }A\subset E \text{ satisfies $\theta(A)\leq 0$}\}.\end{gather*}
\end{lemma}

\begin{proof}
Recall that a torsion pair  is  a pair of full  subcategories $\T,\F\subset \A$ such that

\begin{itemize}
\item[(a)] if $T\in \T$ and $F\in \F$ then $\Hom_\A(T,F)=0$.

\item[(b)] for any $E\in \A$ there is a short exact sequence
\[0\lra T\lra E \lra F\lra 0\]
with $T\in\T$ and $F\in\F$.
\end{itemize}
These properties both follow immediately from the existence of Harder-Narasimhan filtrations and the identity \eqref{no} above.
\end{proof}

\label{late}
Note that if we are only interested in representations of total dimension $\leq k$ then the subcategories $\T(\theta)$ and $\F(\theta)$ are constant in $M_\R\setminus \W_k$. Moreover, at a point of $\W_k$, these categories remain unchanged if we fall off the wall in the $-\delta$ direction:
\begin{equation}
\label{wow}\T(\theta-\epsilon\delta)\cap \A_{\leq k}=\T(\theta)\cap \A_{\leq k}, \quad \F(\theta-\epsilon\delta)\cap \A_{\leq k}=\F(\theta)\cap \A_{\leq k},\quad 0<\epsilon\ll 1.\end{equation}
We shall use the obvious notation
\[1_\T(\theta)=1_{(\halft,1)}(\theta)\in \Hat{G}_\Hall, \qquad 1_\F(\theta)=1_{(0,\halft]}(\theta)\in \Hat{G}_\Hall.\]
It follows from \eqref{readit}, and the proof of Lemma \ref{surj}, that if $\theta_1,\theta_2\in M_\R$ both lie outside the essential support of $\DD$ then
\begin{equation}
\label{haz}\Phi_\DD(\theta_1,\theta_2)=1_{\T}(\theta_2)^{-1}* 1_{\T}(\theta_1)=1_{\F}(\theta_2) * 1_{\F}(\theta_1)^{-1}\in \Hat{G}_\Hall.\end{equation}
Further connections between the Hall algebra scattering diagram and the above torsion pair will appear in Section 8.


\section{Nearby stability conditions}
\label{stab}

Let $(Q,I)$ be a quiver with relations and $\A=\rep(Q,I)$ its category of representations. In this section we explain the relationship between the walls of the Hall algebra scattering diagram described in Section \ref{hallscatter}, and the walls of the second kind in the space of stability conditions on any triangulated category $\D$ in which $\A$ occurs as the heart of a bounded t-structure. Readers unfamiliar with triangulated categories can safely skip this section, together with Sections \ref{blo} below. The material of this section is very much inspired by the work of Nagao \cite[Section 4]{nagao}.

For simplicity we assume throughout that the algebra $\CQ/I$ is finite-dimensional: this is to avoid having to distinguish nilpotent representations from finite-dimensional ones. We fix notation as in Section \ref{quivnot}. In particular $\A$ is the abelian category of representations of a quiver with relations $(Q,I)$.

\subsection{}
Let  us  fix a triangulated category $\D$ equipped with a bounded t-structure whose heart is equivalent to $\A$. The most obvious choice is to take the standard t-structure on the bounded derived category $\D^b(\A)$ but it may  be useful  to allow other possibilities.
%
As  with any bounded t-structure, the inclusion functor gives an identification \[N=K_0(\A)=K_0(\D). \]
In particular, any object $E\in \D$ has a well-defined dimension vector $d(E)\in N$.
%

Given $\theta\in M_\R$ we define a new bounded t-structure by tilting $\A$ with respect to the torsion pair of Lemma \ref{strain}. Explicitly, its heart 
\[\A(\theta)=\<\F(\theta)[1],\T(\theta)\>\subset \D\]
consists of objects $E\in \D$ such that
\[H_\A^{-1}(E)\in \F(\theta),\  H_\A^0(E)\in \T(\theta)\text{ and } H_\A^i(E)=0\text{ otherwise}.\]
The inclusion functor again gives a canonical  identification
\begin{equation}
\label{iddy}K_0(\A(\theta))=K_0(\D)=N.\end{equation}
Recall that the support of the Hall algebra scattering diagram is the subset
\[\W=\bigcup_{k\geq 1} \W_k=\{\theta\in M_\R:\text{there exist $\theta$-semistable objects $0\neq E\in \A$}\}\subset M_\R.\]
We  denote its closure by $\bar{\W}\subset M_\R$.

\subsection{}We refer to the connected components of the complement $M_\R\setminus \bar{\W}$ as \emph{chambers}.

\begin{lemma}
Suppose that $\CC\subset \M_\R\setminus\bar{W}$ is a chamber. Then
\begin{itemize}
\item[(a)] the heart $\A(\theta)\subset \D$ is independent of $\theta\in \CC$ and has finite length.\smallskip

\item[(b)] the subcategory $\A(\theta)$ has finitely many simple objects $T_i$ up to isomorphism, and their classes form a basis of $N$.\smallskip

\item[(c)]  for each $i$ there is a sign $\epsilon(i)\in\{\pm 1\}$ such that $\epsilon(i)\cdot [T_i]\in N^+$.
\end{itemize}
 \end{lemma}

\begin{proof}
Suppose $\theta$ lies in the interior of $M_\R\setminus\W$. The torsion pair $(\T(\theta),\F(\theta))$ is constant in $M_\R\setminus \W$ so we can find $\phi\in M_{\mathbb{Q}}$ with $\A(\phi)=\A(\theta)$. Clearing denominators this means that we can find $\phi\in M$ which is strictly positive on all nonzero elements of $\A(\theta)$. It follows that there can be no infinite descending or ascending chains and so $\A(\theta)$ has finite length. This proves (a).

The classes of a complete set of non-isomorphic  simple objects $T_i\in \A(\theta)$  trivially form a basis for $K_0(\A(\theta))$. Under the identification \eqref{iddy} this corresponds to a basis of $N$, so in particular there are only finitely many $T_i$. This gives (b).

For (c), note that the heart $\A\subset \D$ is the tilt of the heart $\A(\theta)\subset \D$ with respect to the torsion pair $(\F(\theta)[1],\T(\theta))$. Since $T_i\in \A(\theta)$ is simple it must be either torsion or torsion-free. Thus we either have $T_i\in \T(\theta)$ or $T_i[-1]\in \F(\theta)$. Taking $\epsilon(i)=+1$ or $-1$ respectively it follows that $\epsilon(i)\cdot [T_i]\in N^+$.  \end{proof}
 
 In the context of quivers with potential and cluster theory, the classes $[T_i]\in N\isom \Z^{V(Q)}$ of part (b) are known  as  the $c$--vectors corresponding to the chamber $\CC$; the statement of part (c)  then goes under the name of known of sign-coherence.

\subsection{}
We shall now consider stability conditions on the triangulated category $\D$.  We always impose the support condition: there should exist $K>0$ such that for all semistable objects there is an inequality
\[|Z(E)|>K\cdot \| d(E)\|,\]
where $\|\cdot \|$ is some fixed norm on $N_\R$.
It follows from general theory that the set of such stability conditions forms a complex manifold 
$\Stab(\D)$, and that the forgetful map 
\[\Stab(\D)\to M_\C\]
sending a stability condition $(Z,\P)$ to its central charge $Z\colon N\to \C$ is a local homeomorphism.

Let us fix a stability condition $\sigma\in \Stab(\D)$. Any object $0\neq E\in \D$ then has a uniquely-defined collection of Harder-Narasimhan factors $E_1,\cdots, E_m$ with respect to $\sigma$. These are semistable objects in $\sigma$   with descending phases $\phi_1>\cdots>\phi_m$. We set $\phi^+(E)=\phi_1$ and $\phi^-(E)=\phi_m$. The object $E$ is semistable in $\sigma$ precisely if $m=1$ and hence $\phi^+(E)=\phi^-(E)$. For a fixed object $0\neq E\in \D$ the functions \[\phi^\pm(E)\colon \Stab(\D)\to \R\] are continuous. 
Given an interval $I\subset \R$ we write $\P(I)\subset \D$ for the full subcategory consisting of objects which are either 0 or which have the property that all the phases $\phi_i$ of their Harder-Narasimhan factors lie in $I$. In particular, for any $\phi\in \R$, the subcategory $\P(\{\phi\})=\P(\phi)$ consists of the semistable objects of phase $\phi$.

The subcategory $\P(0,1]\subset \D$ is called the heart of the stability condition. It is the heart of a bounded t-structure, and in particular is an abelian category.
A wall of type II in $\Stab(\D)$ is defined to be the locus where a fixed object $0\neq E\in \D$ lies in the subcategory $\P(0)$.  Note that in the complement of the closure of the union of these walls, the heart $\P(0,1]$ is locally constant, since the condition  that a given $0\neq E\in \D$ lies in the heart is that $0<\phi^-(E)\leq \phi^+(E)\leq 1$, and these are simultaneously open and closed conditions.

\begin{defn}
A stability condition $\sigma=(Z,\P)\in \Stab(\D)$ is said to be \emph{nearby} $\A$ if the condition $\A\subset \P(-1,1)$ holds. 
\end{defn}

This  terminology is due to Keller \cite{keller}. We write $\CC(\A)\subset \Stab(\D)$ for the subset of stability conditions nearby $\A$.

\begin{lemma}
The subset $\CC(\A)\subset \Stab(\D)$ is open.
\end{lemma}

\begin{proof}
The condition is just that each of the simple objects $S_i\in \A$ satisfies $-1<\phi^-(S_i)\leq \phi^+(S_i)<1$. Since the functions $\phi^\pm(E)$ are continuous on $\Stab(\D)$
and there are only finitely many of the $S_i$ this is an open condition.
\end{proof}

\subsection{}
\label{alal}
Consider the continuous map
\[F\colon \CC(\A) \to M_\R, \qquad F(Z,\P)=\Im Z,\]
sending a stability condition to the imaginary part of its central charge.

\begin{prop}
The following statements hold:
\begin{itemize}
\item[(a)] The map $F$ is surjective.\smallskip

\item[(b)] If $\sigma\in \CC(\A)$ satisfies $F(\sigma)=\theta\in M_\R$ then $\sigma$ has heart  $\A(\theta)\subset \D$. In particular, stability conditions in a given fibre of $F$ all have the same heart.\smallskip 

\item[(c)] If $\sigma\in \CC(\A)$ satisfies $F(\sigma)=\theta\in M_\R$ then $\P(0)\subset \D$ is the subcategory of $\theta$-semistable objects in $\A$. In particular, the subset $\W\subset M_\R$ is the image of the union of the walls of type II in $\CC(\A)$ under the map $F$. 
\end{itemize}
\end{prop}

\begin{proof}
To prove surjectivity we use the action of the universal cover $\grp$ of the group $\GL^+(2,\R)$ on the space of stability conditions \cite[Section 8]{bridgeland1}. This group can be thought of as the
set of pairs $(T,f)$, where $f\colon \R\to\R$ is an
increasing map with $f(\phi+1)=f(\phi)+1$, and  $T\colon \C\to\C$
is an orientation-preserving $\R$-linear isomorphism, such that the induced maps on
\[S^1=\R/2\Z=\C^*/\R_{>0}\] are the same.
There is a right action of this group on $\Stab(\D)$ in which a 
pair $(T,f)$ maps a stability condition $\sigma=(Z,\P)$  to the stability condition $\sigma'=(Z',\P')$, where $Z'=T\circ Z$ and
$\P'(f(\phi))=\P(\phi)$. 
 Consider the element
\[T=\mat{1}{0}{-1}{1}\in\GL^+(2,\R).\]
It lifts uniquely to an element $(T,f)\in \grp$ such that $f(0,1)\subset (-\halft,1)$.
For any $\theta\in M_\R$ there is a stability condition $\sigma$ on $\D$ with heart $\A$ and central charge $Z=(\delta-\theta)+i\delta$. Indeed, since $\A$ has finite-length, and $\delta(E)>0$ for every nonzero object $E\in \A$,  this follows from the basic  existence result  \cite[Prop. 5.3]{bridgeland1}. Applying the element $(T,f)$ gives a stability condition $\sigma'=(Z',\P')$ such that $\A=\P(0,1)\subset \P(-\halft,1)$ and $\Im (Z')=\theta$. This proves (a).

To prove (b) take a nearby stability condition $\sigma=(Z,\P)\in \CC(\A)$. There is a torsion pair $(\T,\F)\subset \A$ given by
\[\T=\A\cap \P(0,1), \quad \F=\A\cap\P(-1,0].\]
It is a tautology that $\B=\P(0,1]$ is the tilt of $\A$ with respect to this torsion pair. Indeed $\F[1]\subset \B$ and $\T\subset \B$ so the tilted heart is contained in $\B$, and hence, by a standard argument, is
equal to it.
Let $\theta=\Im Z=F(\sigma)$. We claim that $\T=\T(\theta)$ and $\F=\F(\theta)$.
Indeed, any object $E\in \T$ lies in $\P(0,1)$ so  satisfies $\theta(E)> 0$. Similarly, any $E\in \F$ lies in $\P(-1,0]$ so satisfies $\theta(E)\leq 0 $. Since $\T$ and $\F$ are closed under quotient and subobjects respectively (this is true for any torsion pair) this implies that $\T\subset \T(\theta)\text{ and }\F\subset \F(\theta).$ It follows easily that these inclusions are equalities, which 
proves (b).

To prove  (c) we take notation as in the last paragraph and prove that an object $E\in \D$ lies in $\P(0)$ precisely if it lies in $\A$ and is $\theta$-semistable. Suppose first that $E\in \P(0)$. Since $E[1]\in \B=\P(0,1]$, there is a short exact sequence 
\[0\lra X[1]\lRa{f} E[1]\lRa{g} Y\lra 0\]
 in $\B$ with $X\in \F$ and $Y\in \T$. But then  $Y\in \P(0,1)$ which implies that $g=0$. Hence $E=X\in \F\subset \A$. Since $\F=\F(\theta)$ and $\theta(E)=0$ it follows that $E$ is $\theta$-semistable. Conversely, if $E\in \A$ is $\theta$-semistable, then $E\in \F(\theta)$ so that $E\in \B[-1]=\P(-1,0]$ and the fact that $\theta(E)=0$ implies that $E\in \P(0)$.\end{proof}
 
%


\section{Framed representations}
\label{frr}

In this section we consider  representations of quivers equipped with framings, i.e. maps from a fixed projective module. In particular, we introduce certain fine moduli schemes  which  are generalizations of the
moduli spaces studied by Engel and Reineke \cite{er}, and will be used in Section \ref{theta} to describe theta functions associated to the stability scattering diagrams of Section \ref{hallscatter}.

\subsection{}
\label{row}
Let $(Q,I)$ be a quiver with relations. We take notation as in Section \ref{quivnot}.  Associated to each vertex $i\in V(Q)$ is a  projective module  $P_i$. In terms of the idempotent $\epsilon_i\in \CQ$  corresponding to the vertex $i$ it can be written $P_i=(\CQ/I)\epsilon_i$. For each class $m\in M^\oplus$ there is therefore a finitely-generated projective module
\[P(m)=\bigoplus_{i\in V(Q)} P_i^{m(e_i)}.\]
Take $P=P(m)$ of this form.
 We consider the   category $\fr_P(Q,I)$ of $P$-framed representations of $(Q,I)$. The objects are defined to be representations  $E\in \rep(Q,I)$ equipped with a framing map $\nu\colon P\to E$. A morphism between two such $P$-framed representations is  a morphism of the underlying representations which intertwines the framing maps.

 We can give another description of the category  $\fr_P(Q,I)$ as follows. Form a new quiver $Q^\star$ containing $Q$ as a subquiver, by adjoining a new vertex $\star$  and, for each $i\in V(Q)$,   adding $m(e_i)$ arrows from vertex $\star$ to vertex $i$. The two-sided ideal of relations $I\subset \C Q$ generates a  two-sided ideal $I^\star\subset \C Q^\star$. Let \begin{equation}
 \label{subc}\rep(Q^\star,I^\star)_1\subset \rep(Q^\star,I^\star)\end{equation}
 be the subcategory whose objects are representations having the one-dimensional vector space $\C$ at the  vertex $\star$, and whose morphisms are morphisms of representations which have the identity map $\id\colon \C\to \C$ at the  vertex $\star$.

\begin{lemma}
\label{trivial}
There is an equivalence of categories \[\fr_P(Q,I)\isom \rep(Q^\star,I^\star)_1.\]
In particular, isomorphism classes of $\fr_P(Q,I)$ are in bijection with isomorphism classes of $\rep(Q^\star,I^\star)$ having a one-dimensional vector space at the extending vertex.
\end{lemma}

\begin{proof}
Take $E\in \rep(Q,I)$ and let $E_i$ be the vector space at vertex $i\in V(Q)$. There is a canonical isomorphism
\[\Hom_{\CQ/I} (P_i,E)\isom E_i,\]
where $P_i=P(e_i^*)$ is the  projective corresponding to vertex $i$.
Choosing an identification $P=P(m)$ we obtain a canonical isomorphism
\[\Hom_{\CQ/I} (P,E)\isom \bigoplus_{i\in V(Q)} E_i^{\oplus m(e_i)}.\]
The vector space on the right parameterizes precisely the data required to extend $E$ to a representation of $(Q^\star,I^\star)$ with $\C$ at the extending vertex.
\end{proof}

The important point for us will be that framed representations of $(Q,I)$ can be viewed as representations of the extended quiver $(Q^\star,I^\star)$. But when it comes to considering moduli stacks we must be careful: the subcategory \eqref{subc} is not full, and a framed representation can have no non-trivial automorphisms (this happens for example if the framing map $\nu$ is surjective), whereas a representation of $(Q^\star,I^\star)$ always has at least a $\C^*$ group of automorphisms.

\subsection{}

Fix a class $m\in M^\oplus$ and consider again the corresponding extended quiver with relations $(Q^\star,I^\star)$.
 We let \[N^\star:=\Z^{V(Q^\star)}=\Z^{V(Q)}\oplus \Z=N\oplus \Z, \qquad M^\star:=\Hom_\Z(N^\star,\Z)=M\oplus\Z.\]
 For any dimension vector $d\in N$ we define $d^\star=(d,1)\in N^\star$. Given $\theta\in M_\R$ and $d\in N$ there is a unique lift $\theta^\star\in M^\star_\R$ such that $\theta^\star(d^\star)=0$. Explicitly we have
\[\theta^\star=(\theta,-\theta(d))\in M\oplus \Z.\]
We denote by $M^{\ss} (d^\star,\theta^\star)$ the coarse moduli scheme  of $\theta^*$-semistable representations of $(Q^\star,I^\star)$ of dimension vector $d^\star$.

\begin{lemma}
\label{7c}
The moduli scheme $M^{\ss} (d^\star,\theta^\star)$ is  fine providing  $\theta$ does not lie in the subset $\W_k\subset M_\R$ of Section \ref{sixa}, where $k=\delta(d)$ is the total dimension of $d$.
\end{lemma}

\begin{proof}
Since the dimension vector $d^\star\in N^\star$ is clearly primitive, the moduli space  will be fine providing that there are no strictly  $\theta^\star$-semistable objects \cite[Prop. 5.3]{king}. 
Suppose \[0\lra A\lra E\lra B\lra 0\]
 is a short exact sequence in $\rep(Q^\star,I^\star)$ such that $E$ has dimension vector $d^\star=(d,1)$.  Then either $A$ or $B$ has dimension vector of the form $(d',0)$. Thus if $E$ is strictly $\theta^\star$-semistable, it must have a $\theta^\star$-stable factor $S$ of dimension vector $(d',0)$. Then $S$ is  a non-zero $\theta$-stable representation  of $(Q,I)$ and hence $\theta\in\W_k$.
\end{proof} 

The proof of Lemma \ref{7c} shows that  $M^{\ss} (d^\star,\theta^\star)$ is constant as $\theta$ varies in $M_\R\setminus\W_k$. For a general point of 
$\theta\in M_\R$ we set
\[F(d,m,\theta)=M^{\ss} (d^\star,(\theta-\epsilon\delta)^\star), \quad 0<\epsilon\ll 1.\]
Since $\delta$ does not lie in $\W_k$,  this is always a fine moduli scheme. We set
\[F(m,\theta)=\bigsqcup_{d\in N^\oplus} F(d,m,\theta).\]
There is an obvious morphism $r\colon F(m,\theta)\to \M$ sending a semistable representation of $(Q^\star,I^\star)$ to the representation of $(Q,I)$ obtained by restriction.

\subsection{}
Let $P=P(m)$ be the projective representation of $(Q,I)$ corresponding to a class $m\in M^\oplus$, and fix also a class $\theta\in M_\R$.  The following result shows that the  moduli scheme $F(m,\theta)$ can be viewed as parameterizing $P$-framed representations of $(Q,I)$ of a particular type.


\begin{lemma}
\label{stung}Under the correspondence of Lemma \ref{trivial}, a framed representation $\nu\colon P\to E$ of dimension vector $d\in N^\oplus$ corresponds to a point  of $F(d,m,\theta)$ precisely if
\begin{itemize}
\item[(a)]  $E\in \F(\theta)$,
\smallskip

\item[(b)] $\coker(\nu)\in \T(\theta)$.
\end{itemize}
\end{lemma}

\begin{proof}
The definition of $F(d,m,\theta)$ together with \eqref{wow} shows that we can assume that $\theta$ does not lie in $\W_k\subset M_\R$ where $k=\delta(d)$. 
It remains to show that  a $(Q^\star,I^\star)$ representation  $E^\star$ of dimension vector $d^\star$ is $\theta^\star$-semistable, precisely if the corresponding framed representation $(E,\nu)$ of  $(Q,I)$ satisfies  $E\in \F(\theta)$ and $\coker(\nu)\in \T(\theta)$.
Consider a short exact sequence of representations of $(Q^\star, I^\star)$
\[0\lra A^\star\lra E^\star\lra B^\star\lra 0.\]
There are two possibilities: either $A^\star$ or $B^\star$ has dimension vector $(n,0)$  for some $n\in N$, and hence is  a representation of $(Q,I)$ extended by zero.

In the first case $A^\star$ is a representation $A$ of $(Q,I)$ extended by zero. Moreover,  $A$  is a $(Q,I)$-subrepresentation of $E$. Conversely, any subrepresentation $A\subset E$ can be extended by zero to give a subrepresentation $A^\star\subset E^\star$. Then  $A^\star$ destabilizes $E^\star$ precisely if $\theta(A)\geq 0$. Since $\theta\notin\W_k$  this in fact implies $\theta(A)>0$.  But the existence of such a subrepresentation  $A\subset E$ is precisely the condition that $E\notin \F(\theta)$. 

In the second case $B^\star$ is a representation $B$ of $(Q,I)$ extended by zero. We therefore obtain a surjective map of $(Q,I)$-representations $g\colon E\to B$. But the fact that this extends to a map $E^\star\to B^\star$ forces the composite $g\circ \nu$ to be zero. Conversely, when $g\circ \nu=0$, any map $g$ extends uniquely to $g^\star\colon E^\star\to B^\star$. We conclude that quotients $E^\star\onto B^\star$ of dimension vector $(0,n)$ correspond to quotients $B$ of $\coker(\nu)$ of dimension vector $n$. We can destabilize $E^\star$ in this way precisely if we can find such a $B$ with $\theta(B)\leq 0$. This is equivalent to the condition  that $\coker(\nu)\notin \T(\theta)$.
\end{proof}

\subsection{}
\label{blo}
We finish this section by giving yet another interpretation of the framed quiver moduli spaces $F(d,m,\theta)$, this time in terms of the tilted hearts of Section \ref{stab}. This result will not be used later, and can be safely  skipped.  
For simplicity we shall assume that $\CQ/I$ is finite-dimensional so that all finitely-generated projective representations lie in the category $\A=\rep (Q,I)$.
  
Given $\theta\in M_\R$ we can consider the corresponding torsion sequence
\[0\lra R(m,\theta)\lra P(m)\lra U(m,\theta)\lra 0\]
with $R(m,\theta)\in \T(\theta)$ and $U(m,\theta)\in \F(\theta)$.
Also introduce the shifted heart
\[\B(\theta)= \A(\theta)[-1]=\<\F(\theta), \T(\theta)[-1]\>\subset \D.\]
Note that $U(m,\theta)\in \B(\theta)$.

%

\begin{prop}
\label{last}
There is a natural bijection between  points of the scheme
 $F(d,m;\theta)$ and isomorphism classes of pairs $(E,\mu)$ where
\begin{itemize}
\item[(a)] $E\in \B(\theta)$ is an object of class $[E]=d\in N$,
\item[(b)] $\mu\colon U(m,\theta)\to E$ is a surjection in $\B(\theta)$.
\end{itemize}
\end{prop}

\begin{proof}
This follows easily from Lemma \ref{stung}.  Given a pair $(E,\mu)$ as in the statement, we can form a short exact sequence in $\B(\theta)$
\[0\lra K\lra U(m,\theta)\lra E\lra 0\]
Now take the associated long exact sequence in cohomology with respect to the  t-structure with heart $\A$. This shows  that $E=H_\A^0(E)$ and hence that $E\in \F(\theta)$. Moreover, the cokernel of $\mu$, viewed as a map in $\A$, is  the object $H^1_\A(K)\in \T(\theta)$. Composing with the  surjection  $P(m)\to U(m,\theta)$ we  get a map $\nu\colon P(m)\to E$ in the category $\A$ whose cokernel lies in $\T(\theta)$.

For the converse, suppose given a map $\nu\colon P(m)\to E$ in $\A$ such that $E\in\F(\theta)$ and whose cokernel lies in $\T(\theta)$. Since $R(m,\theta)\in \T(\theta)$ this factors via $P(m)\onto U(m,\theta)$ and hence induces a map $\mu\colon U(m,\theta)\to E$. If $\mu$ is not surjective in $\B(\theta)$ we can find a surjection $q\colon E\onto Q$ such that $q\circ \mu=0$. The long exact sequence in $\A$-cohomology as above then implies that $Q\in \F(\theta)$. This then contradicts  $\coker(\nu)\in\T(\theta)$.
\end{proof}

It can often happen that the category $\B(\theta)$ is equivalent to a category of the form $\rep(Q',I')$ for some new quiver $(Q',I')$. In that case, the proof of Prop. \ref{last}  can easily be extended to show  that the scheme $F(d,m,\theta)$ is a quiver Grassmannian.


\section{Stacks of framed representations}

In this section we consider stacks parameterizing framed representations of the various kinds considered in Section \ref{fr}. We then prove some important identities relating the corresponding elements of the motivic Hall algebra. Throughout $(Q,I)$ is a fixed quiver with relations and $P=P(m)$ is a  finitely-generated projective module for $\CQ/I$ corresponding to a class $m\in M^\oplus$.

\subsection{}
\label{nag}
 We begin by introducing the stack $\mfr_P$
of $P$-framed representations of $(Q,I)$. The reader is advised to first recall the definition of the stack $\M$ from Section \ref{blib}.
For each scheme $S$ over $\C$ we first consider the quasi-coherent $\O_S$-module \[\P_S:=P\tensor_\C \O_S\]
 with the induced action of  $\CQ/I$. This is just the pullback of the representation $P$ via the projection $S\to \Spec(\C)$ exactly as in the definition of the stack $\M$, except that since $P$ may be infinite-dimensional in general, the resulting locally-free sheaf may have infinite rank.
 
 The objects of the stack $\mfr_P$ over a scheme $S$  consist of an object $(\E,\rho)$ of $\M(S)$ together with a morphism of quasi-coherent sheaves
$\nu\colon \P_S\to \E$ which intertwines the actions of $\CQ/I$ on $\P_S$ and $\E$. Given another object $(\E',\rho',\nu')$ lying over a scheme $S'$, a morphism in $\M_P$
\[(\E',\rho',\nu')\to (\E,\rho,\nu)\]
 lying over a morphism  $f\colon S'\to S$ is defined to be an isomorphism of $\O_S$-modules $\theta\colon f^*(\E)\to \E'$ intertwining the maps $\rho$ and $\rho'$ as in the definition of the stack $\M$, with the further condition that this isomorphism commutes with the $\O_S$-module maps $\nu$, $\nu'$ in the obvious way.

 We leave it to the reader to check the stack conditions, which follow easily from those for $\M$. There is an obvious morphism of stacks
$q\colon \mfr_P\to \M$
obtained by forgetting the framings.

\begin{lemma}
\label{8a}
Restricted to the open and closed subset $\M_d\subset \M$, the map $q$ is a vector bundle of rank $m(d)$.
\end{lemma}

\begin{proof}
Let $S$ be a scheme and $f\colon S\to \M_d$ an $S$-valued point corresponding to a representation $(\E,\rho)$ of $(Q,I)$ over $S$. Pulling the morphism $q$ back to $S$ gives a stack $S_P\to S$ whose $T$-valued points consist of a morphism $g\colon T\to S$ and a framing of $g^*(\E,\rho)$. Now by the argument of Lemma \ref{trivial}, the set of such framings coincides with the set of sections of the sheaf $g^*(\V)$ on $T$, where
\[\V=\bigoplus_{i\in V(Q)} \E_i^{\oplus m(i)}\]
Note that $\V$ is a locally-free sheaf on $T$ of rank $m(d)$.
 It follows that the stack $S_P$ is represented by the total space of the associated vector bundle, which proves the result.\end{proof}

%
\subsection{}
 Globalising Lemma \ref{stung} we would like to identify the scheme $F(m,\theta)$ with an open substack of $\M_P$. There are two slightly tricky issues to understand first:
\smallskip

\begin{itemize}
\item[(i)] Let $\M^\star_1$ denote the stack of representations of the extended quiver $(Q^\star,I^\star)$ having a one-dimensional representation at the extending vertex. A family version of Lemma \ref{trivial} (see also the proof of Lemma \ref{8a}) gives a morphism of stacks
$\M_P\to \M^\star_1$. However, because of the discrepancy in automorphism groups referred to in Section \ref{row} this is not an isomorphism, but rather a   $\C^*$-bundle.
\smallskip

\item[(ii)] We wish to define an open substack of $\M_P$ whose points parameterize framed objects $\nu\colon P\to E$ for which $E\in \F(\theta)$ and $\coker(\nu)\in \T(\theta)$. However, when working with objects of $\M_P$ over a general scheme $S$, there is no reason to expect the cokernel of the framing map $\nu\colon \P_S\to \E$ to be locally-free and hence define an object of $\M(S)$. So the definition of such a substack is problematic. 

\end{itemize}

The first issue does not in fact cause a problem,  since the scheme $F(m,\theta)$ is also a $\C^*$-bundle over the corresponding open substack of $\M^\star_1$ for the usual reason that moduli stacks of stable objects are $\C^*$-gerbes over the corresponding moduli schemes.

To deal with  the second issue we define a new stack $\M'_P$ in the same way as $\M_P$ but imposing the condition that the cokernel of the framing map $\nu$ should be locally-free. This is indeed a substack because forming cokernels commutes with restriction to open subsets, and the condition that a sheaf be locally-free can be checked locally. The inclusion morphism $\M'_P\to \M_P$ is then an equivalence  on $\C$-valued points, so that from the point-of-view of the motivic Hall algebra the two stacks are interchangeable.

\subsection{}
The stack $\M'_P$  has two obvious morphisms
\[p\colon \M'_P\to \M, \quad q\colon \M'_P\to \M.\]
The first morphism $p$ takes a framed representation $\nu\colon \P_S\to \E$ to the finite-rank locally-free sheaf $\coker(\nu)$, with the obvious induced  action of $\CQ/I$. The second morphism $q$ forgets the framings as before: it takes $\nu\colon \P_S\to \E$ to the locally-free sheaf $\E$ with its given action of $\CQ/I$. 
Given $\theta\in M_\R$ we define an open substack
\[\M'_P(\F(\theta))=\M'_P\cap q^{-1}(\F(\theta)), \quad \M'_P(\T(\theta),\F(\theta))=\M'_P(\F(\theta))\cap p^{-1}(\T(\theta)) \]
The following is a moduli-theoretic version of Lemma \ref{stung}.

\begin{lemma}
\label{abo}
There is a morphism of stacks
\[\M'_P(\T(\theta),\F(\theta)) \lra F(m,\theta)\]
which induces an equivalence on $\C$-valued points.
\end{lemma}
 
 \begin{proof}
We have morphisms of stacks $\M'_P\to \M_P\to \M^\star_1$. So given an $S$-valued point of the stack $\M'_P(\T(\theta),\F(\theta))$ we obtain a family of $(Q^\star,I^\star)$ representations over $S$. Lemma \ref{stung} shows that these in fact define a family of stable representations, and thus give rise to a map $S\to F(m,\theta)$. This defines the required morphism. Lemma \ref{stung} again shows that it is an equivalence on $\C$-valued points.
 \end{proof}

\subsection{}
Let us again fix an element $\theta\in M_\R$. Using the remark of  Section \ref{spelement}, we obtain elements
\[ 1^P_\F(\theta)=[\M'_P(\F(\theta))\lRa{q} \M], \qquad F(m,\theta) = [F(m,\theta)\lRa{r} \M],\]
in the completed Hall algebra $\Hat{H}(Q,I)$. 

\begin{lemma}
\label{cloud}
In the algebra $\Hat{H}(Q,I)$  there is an identity
\[1^P_\F(\theta)=F(m,\theta)*1_\F(\theta).\]
\end{lemma}

\begin{proof}
By Lemma \ref{abo} we can replace $F(m,\theta)$ by the stack $\M'_P(\T(\theta),\F(\theta))$. The product on the right of the identity is then represented by a stack $X$ whose objects over a scheme $S$ consist of short exact sequences
\[0\lra \E_1\lRa{\alpha} \E_2\lRa{\beta} \E_3\lra 0\]
of families of quiver representations over $S$, together with a framing $\nu\colon \P_S\to \E_1$ whose cokernel is locally-free. We also insist that restricted to any $\C$-valued point, all representations lie in $\F(\theta)$, and the cokernel of  $\nu$ lies in $\T(\theta)$. Sending such an object to  the family $\E_2$ equipped with the framing $\alpha\circ \nu$ defines a morphism of stacks  \[h\colon X\to \M'_P(\F(\theta)).\]
 The identity will follow if we show that $h$ induces an an equivalence on $\C$-valued points.
To prove this, note that given an element $E\in \F(\theta)$ and a framing $\nu\colon P\to E$, there is a unique short exact sequence
\begin{equation}
\label{noo}0\lra A\lRa{f} E\lRa{g} B\lra 0\end{equation}
such that $A,B\in \F(\theta)$, the composite $g\circ\nu=0$, and the induced map $\nu\colon P\to A$  has cokernel in $\T(\theta)$. Indeed, the quotient $B$ is obtained uniquely as the torsion-free part of $\coker(\nu)$.
\end{proof}

Let us introduce further elements of $\Hat{H}(Q,I)$ as follows.
\[ \qquad 1^P_\ss(\theta)=[\M'_P\cap q^{-1}(\M_{\ss}(\theta))\lRa{q} \M], \quad F_\ss(m,\theta) = [F(m,\theta)\cap r^{-1}(\M_{\ss}(\theta))\lRa{r} \M].\]
 The following is an easy consequence of Lemma \ref{cloud}.

\begin{lemma}
\label{clouds}
In the algebra $\Hat{H}(Q,I)$  there is an identity
 \[1^P_\ss(\theta)=F_\ss(m;\theta) * 1_\ss(\theta).\]
\end{lemma}

\begin{proof}
Note that if $E\in \F(\theta)$ then $E$ is $\theta$-semistable precisely if $\theta(E)=0$. 
Thus given a short exact sequence  \eqref{noo} with all objects in $\F(\theta)$ then \[\theta(E)=0 \iff \theta(A)=\theta(B)=0.\] The identity then follows in the same way as Lemma \ref{cloud} or by simply restricting the identity of Lemma \ref{cloud} to the substack of  representations whose dimension vectors are orthogonal to $\theta$.
\end{proof}


 \section{Theta functions}
\label{theta}

The theta functions defined by a scattering diagram are of crucial importance in the applications  to cluster varieties described in \cite{ghkk}.  In the case of the stability scattering diagram, we can describe some of these theta functions in terms of generating functions for Euler characteristics of the framed quiver moduli introduced in Section \ref{frr}.  The general case is left for future research.

\subsection{}
\label{rivelin}
 We begin by considering the following abstract set-up. Take a finite-rank free abelian group $N$ with a positive cone $N^+\subset N$  as in Section \ref{basic} and use notation as defined there. Now take $A$ to be an $N^\oplus$-graded Poisson algebra
\[A=\bigoplus_{n\in N^\oplus} A_n, \quad A_{n_1}\cdot A_{n_2}\subset A_{n_1+n_2}, \quad
\{A_{n_1},A_{n_2}\}\subset A_{n_1+n_2}.\]
 For each $k\geq 1$ we can consider the Poisson ideal $A_{>k}=\oplus_{\delta(n)>k} \, A_n,$ and the corresponding quotient $A_{\leq k}=A/A_{\geq k}$. Taking the limit of the obvious maps gives a Poisson algebra 
 \[\Hat{A}=\lim_{\longleftarrow} \;A_{\leq k} .\]
  The subspace $\g=A_{>0}\subset A$ is a Lie algebra under the Poisson bracket. The corresponding completion $\Hat{\g}=\Hat{A}_{>0}$ is a pro-nilpotent Lie algebra. As usual, we denote the corresponding pro-unipotent group by $\Hat{G}$.
  
  \subsection{}
We shall also need to equip the vector space $B=A\tensor_\C \C[M]$ with the structure of an $N^\oplus$ graded Poisson algebra in such a way that the inclusions
\[A\hookrightarrow B,\quad a\mapsto a\tensor 1, \qquad \C[M]\hookrightarrow B, \quad z^m\mapsto 1\tensor z^m\]
are maps of graded Poisson algebras (we equip the algebra $\C[M]$ with the trivial Poisson bracket and consider all elements to have degree 0). The same completion process then gives a  Poisson algebra
\[\Hat{B}=\Hat{A}\tensor_\C \C[M].\]
 
 Since $\Hat{A}$ as a subalgebra of $\Hat{B}$ closed under Poisson bracket,
 the Lie algebra $\Hat{\g}=\Hat{A}_{>0}$ acts on  $B$ by derivations via $a(b) =\{a,b\}$.
Exponentiating gives an action of the group $\Hat{G}$  by Poisson algebra automorphisms of $\Hat{B}$:
\begin{equation}
\label{form}\exp(x)(b)=\exp\{x,-\} (b).\end{equation}
This is the action we will use to define our theta functions.

%

\subsection{} {\it Example.}\label{egg1}
 Let $N$ be a finite rank free abelian group with a skew-symmetric form \[\<-,-\>\colon N\times N \to \Z.\]
  Also fix a basis $(e_1,\cdots,e_n)$ and let $N^+\subset N$ be the positive cone spanned by the $e_i$. Let $A$  be the monoid algebra \[\C[N^\oplus]=\bigoplus_{n\in N^\oplus} \C\cdot x^n,\] and equip it with the Poisson bracket
 \[\{x^{n_1},x^{n_2}\}=\<n_1,n_2\> \cdot x^{n_1+n_2}.\]
 We take $B$ to be the commutative tensor product algebra $B=\C[N^\oplus]\tensor_\C \C[M]$
 with 
  \[\{x^n,z^m\}=m(n)\cdot x^n\cdot  z^{m}, \quad \{z^{m_1},z^{m_2}\}=0.\]
  Writing $x_i=x^{e_i}$ and $z_i=z^{e_i^*}$ we have
 \[A=\C[x_1,\cdots,x_n], \quad B=\C[x_1,\cdots,x_n][z_1^{\pm 1}, \cdots, z_n^{\pm 1}].\]
  Completing with respect to the $N^\oplus$-grading gives Poisson algebras
 \[\Hat{A}=\C[[x_1,\cdots,x_n]], \quad \Hat{B}=\C[[x_1,\cdots,x_n]][z_1^{\pm 1}, \cdots, z_n^{\pm 1}].\]
 This is the relevant context  for the stability and cluster scattering diagrams described in the Introduction.
%

\subsection{}{\it Example.}\label{egg2}
Take a quiver with potential $(Q,I)$ and take $A$ to be the Hall algebra\[A=H(Q,I)=\bigoplus_{d\in N^\oplus} H(Q,I)_d,\] over $\C(t)$. We define an algebra structure on  $B=A\tensor_\C \C[M]$ extending that on the subalgebras $A$ and $\C[M]$ by setting
\[a_d * z^m = t^{2m(d)} \cdot z^m* a_d, \quad a_d\in H(Q,I)_d.\]
We equip the algebras $A$ and $B$ with the scaled commutator bracket
\[\{a,b\}=(t^2-1)^{-1}\cdot [a,b]\]
as in \eqref{brack}. In particular, we have
 \[\{a_d,z^m\}=\frac{t^{2m(d)}-1}{t^2-1}\cdot a_d * z^m, \quad \{z^{m_1},z^{m_2}\}=0.\]
This is the context relevant to the stability scattering diagram.

Note that as in Section \ref{puff}, the subspace $\g=A_{>0}$ viewed as a Lie algebra via its Poisson bracket is isomorphic to the Lie algebra $\g_{\Hall}$ via the map $x\mapsto (t^2-1)^{-1}\cdot x$. Thus the corresponding group $\Hat{G}$ can be identified with $\Hat{G}_{\Hall}$.

\subsection{}Returning to the general context of Sections 10.1--2, let    $\DD=(\DD_k)_{k\geq 1}$  be a scattering diagram in $\g=A_{>0}$. Recall the associated elements
$\Phi_\D(\theta_1,\theta_2)\in \Hat{G}$
from Section \ref{threeone}.
For each $m\in M$ we define\footnote{Theta functions can also be defined for more general classes $m\in M$ \cite[Section 3]{ghkk} but we shall not consider those here since we are unable to provide a moduli-theoretic description for them.}  a theta function

\[\vartheta^m\colon M_\R\setminus \supp(\DD) \lra \Hat{B}\]
by using the action of the group $\Hat{G}$ on the  algebra $\Hat{B}$, and writing
\[\vartheta^m(\theta)=\Phi_\DD(\theta_+,\theta)(z^m),\]
for an arbitrary element $\theta_+\in M_\R^+$.

Consider composing $\vartheta^m$ with the projection $\pi\colon \Hat{B}\to B_{\leq k}$.  The result extends to a function \[\vartheta^m\colon M_\R\setminus \supp(\DD_k) \lra B_{\leq k}\]
which   is constant on each connected component of the domain. On crossing a generic point of a wall $\dd\in \DD_k$ in the positive direction,  the theta function changes by the action of the corresponding wall-crossing automorphism:
\[\vartheta^m (\theta)\mapsto \Phi_{\DD_k}(\dd)\big(\vartheta^m(\theta)\big)\in B_{\leq k}.\]
 Note  also that by definition $\vartheta^m(\theta)=z^m$ for all  $\theta\in M_\R^+$.

\subsection{}
Consider the case of the Hall algebra scattering diagram of Theorem \ref{hallsc}. Thus we take $A=H(Q,I)$ equipped with the Poisson bracket \eqref{brack} and consider the extended algebra $B=A\tensor_{\C} \C[M]$ as in Section \ref{egg2}.

 \begin{thm}
 \label{mool}
Let $\DD$ be the Hall algebra scattering diagram for the pair $(Q,I)$ and fix $m\in M^\oplus$. Then  there is an identity
\[\vartheta^m(\theta)=z^{m}\cdot F(m,\theta)\in \Hat{B}.\]
 \end{thm}
 
 \begin{proof}
Take an element $m\in M^\oplus$. Note that by \eqref{haz} the Hall algebra scattering diagram satisfies
\[\vartheta^m(\theta)=\Phi_\DD(\theta_+,\theta)(z^m)=1_\F(\theta)(z^m)\in \Hat{B}.\]
The Lie algebra of the group $\Hat{G}_{\Hall}$ can be identified with $A_{>0}$ equipped with the Poisson bracket of Section \ref{egg2}. Alternatively it can be identified with $A_{>0}$ equipped with the commutator bracket. These two identifications differ by a factor of $(t^2-1)$. Using the second identification we see that  as in Section \ref{dull}, the action of the element $1_\F(\theta)\in \Hat{G}_{\Hall}$ on $\Hat{B}$ is by conjugation. But
Lemmas \ref{8a} and \ref{cloud} show that there is an identity in $\Hat{B}$
 \[ 1_\F(\theta)\cdot z^m= z^m\cdot 1^P_\F(\theta)=z^m\cdot F(m,\theta) * 1_\F(\theta).\]
 This proves the result.
 \end{proof}

A very similar proof gives an alternative description of the wall-crossing automorphisms of the Hall algebra scattering diagram.

 \begin{thm}
 \label{mool2}
Let $\DD$ be the Hall algebra scattering diagram for the pair $(Q,I)$.   Then
the action of   $\Phi_\DD(\dd)$ on $\Hat{B}$ at a general point of the support of $\DD$ satisfies
\[z^m\mapsto z^m \cdot \Big[\big.\coprod_{\theta(d)=0}F(d,m,\theta)\to \M \Big].\]
 \end{thm}
 
 \begin{proof}
Apply the proof of  Theorem \ref{mool} replacing  Lemma \ref{cloud} with Lemma \ref{clouds}.
\end{proof}


\section{Quivers with potential: the CY$_3$ case}
\label{load}

In this section we consider the three-dimensional Calabi-Yau situation, when $Q$ is a 2-acyclic quiver with finite potential $W$.  We apply a homomorphism of Lie algebras to the Hall algebra scattering diagram of Section \ref{hallscatter} to obtain a  more concrete scattering diagram which we call the stability scattering diagram of the pair $(Q,W)$. 

There is another scattering diagram associated to $Q$ which we call the cluster scattering diagram associated to $Q$. It was first introduced by Kontsevich and Soibelmnan and plays a vital role in the recent work of Gross, Hacking, Keel and Kontsevich \cite{ghkk}. We show that when $Q$ is acyclic this scattering diagram coincides with the stability scattering diagram.

\subsection{} Let $(Q,I)$ be a quiver with relations as in Section \ref{quivnot}. We now assume that the relations $I\subset \C Q$ are defined by a polynomial potential
\[I=(\partial_a W :a\in A(Q))\subset \C Q.\]
We always assume that $Q$ is 2-acyclic which implies that the potential $W\in \C Q_{\geq 3}$ is a finite sum of cycles of length $\geq 3$.
 The quotient algebra $\C Q/I$ is called the  Jacobi algebra of the pair $(Q,W)$.
Define a skew-symmetric form
 \[\<-,-\>\colon N\times N\to \Z\]
 by setting 
$\<e_i,e_j\>=a_{ji}-a_{ij}$,
where $a_{ij}$ is the number of arrows in $Q$ from vertex $i$ to vertex $j$. In the case that the category $\A=\rep(Q,I)$ is \CY this will coincide with the Euler form on $N=K_0(\A)$ but it is not necessary to assume this. 

Exactly as in Section \ref{egg1} we define a Poisson bracket on $A=\C[N^\oplus]$ by setting
\[\{x^{n_1},x^{n_2}\}=\<n_1,n_2\> \cdot x^{n_1+n_2}.\]

The following crucial result allows us to transport  statements in the Hall algebra into statements in the much simpler algebra $\C[N^\oplus]$. 

\begin{thm}
\label{mine}
There is a homomorphism of $N^\oplus$-graded Poisson algebras
 \[I\colon H_{\reg}(Q,W)\to \C[N^\oplus], \quad I([f\colon X\to \M_d])=e(X)\cdot x^d,\]
 where $e(X)$ denotes the Euler characteristic of  the finite-type complex scheme $X$ equipped with the analytic topology.
\end{thm}

\begin{proof}
The version of this statement for Hall algebras of coherent sheaves on Calabi-Yau threefolds is the special case of \cite[Theorem 5.1]{bridgeland2}  corresponding to the constant constructible function $\one\colon \M\to \Z$. Exactly the same proof works in the case of Hall algebras of quiver representations. The important point to note is that for any pair of objects $E,F\in \A$, the difference
\[\big(\dim_\C \Hom_\A(E,F)-\dim_\C\Ext_\A^1(E,F)\big)-\big(\dim_\C \Hom_\A(F,E)-\dim_\C \Ext^1_\A(F,E)\big),\]
is given by the form $\<[E],[F]\>$.
This is because we can consider $\A$ as the heart of a bounded t-structure in a \CY triangulated category, namely the bounded derived category of the Ginzburg algebra, and the above expression is the Euler form of this category.
\end{proof}

 \subsection{}As usual the subspace $A_{>0}=\C[N^+]$ is a Lie algebra $\g$ under Poisson bracket.
 The homomorphism  $I$ of Theorem \ref{mine} immediately  induces homomorphisms
 \[I\colon \Hat{\g}_\reg \to \Hat{\g}, \qquad I\colon \Hat{G}_\reg \to \Hat{G}\]
  of the corresponding pro-nilpotent Lie algebras and pro-unipotent groups. We remarked in Section \ref{bustfinger} that the Hall algebra scattering diagram  takes values in the Lie subalgebra $\g_{\reg}\subset \g_{\Hall}$. Applying the Lie algebra homomorphism $I$  therefore  gives a scattering diagram in $\g$. 
  
  Theorem \ref{main} now follows immediately from Theorem \ref{hallsc}. Recall that according to Theorem \ref{biggy}, the substack of $\theta$-semistable objects of $\A$ defines an element \[1^\ss(\theta)\in \Hat{G}_\reg\subset \Hat{G}_{\Hall}.\]
 Applying the group homomorphism $I$
gives  a corresponding element of $\Hat{G}$. Explicitly this takes the form
\[I(1^\ss(\theta))=\exp\big(\sum_{d\in N^+}  J(d,\theta)\cdot x^d\big),\]
 where the numbers $J(d,\theta)\in \Q$ are what we referred to in the introduction as  \emph{Joyce invariants}. They are defined to be the Euler characteristics  of the elements  $\epsilon(d,\theta)\in H_{\reg}(Q,I)_d$ appearing in the decomposition \[1^\ss(\theta)=\exp\Big((t^2-1)^{-1} \cdot \sum_{d\in N^+}  \epsilon(d,\theta)\Big)\]
corresponding to the identification \eqref{sub}.

\subsection{}
We now introduce extended Poisson algebras
\[B_{\reg}=H_{\reg}(Q,I)\tensor_\C \C[M], \quad B=\C[N^{\oplus}]\tensor_\C \C[M],\]
exactly as in Sections \ref{egg1} and \ref{egg2}. Comparing the products and Poisson brackets, and noting that the map $I$ sends  $t$ to $1$ it is clear that $I$ extends to give a Poisson algebra map
\[I\colon B_{\reg}\to B,\]
and a similar map on completions.
Applying this to Theorem \ref{mool}  gives Theorem \ref{aut}.

Applying the map $I$ to Theorem \ref{mool2} shows that the action of $\Phi_\DD(\dd)$ on $\Hat{B}$ at a general point of a wall $\dd\subset \DD$ satisfies
\begin{equation}\label{climbthissummer?}z^m\mapsto z^m \cdot \sum_{\theta(d)=0}K(d,m,\theta) \cdot x^d.\end{equation}
A
trivial but important observation  is that for any $n\in N$, the element
 \[ x^n\cdot \prod_{i\in V(Q)} z_i^{\<n,e_i\>}\in B\]
is invariant under  the action of $\Hat{G}$, because for any $d\in N$ 
\[\<d,n\>+\sum_{i\in V(Q)} e_i^*(d)  \cdot \<n,e_i\> =\<d,n\>+\<n,d\>=0.\]
It follows that the action of an element of $\Hat{G}$ on $\Hat{B}$ is determined by its action on the elements $z^i$.  Applying this observation to \eqref{climbthissummer?} gives Theorem \ref{mool2}.

\subsection{}
Let $Q$ be a 2-acyclic quiver  and take notation as above. Assume that the form $\<-,-\>$ is non-degenerate. Let $\DD$ be an arbitrary scattering diagram taking values in $\g$. Any wall $\dd$ of $\DD$ is contained in a hyperplane $n^\perp$ for a unique primitive element $n\in N^+$. We say that $\dd$ is incoming  if it contains the vector \[\theta_n=\<-,n\>\in M.\] Kontsevich and Soibelman proved that a consistent scattering diagram taking values in $\g$ is uniquely specified up to equivalence by its set of incoming walls and their associated wall-crossing automorphisms.
  A particular case of this is 

\begin{thm}[Kontsevich-Soibelman]
For any quiver $Q$ there is a consistent scattering diagram $\DD$  taking values in $\g$   such that the only incoming walls are the hyperplanes $\dd_i=e_i^\perp,$ with associated wall-crossing automorphisms \[\Phi_\DD(\dd_i)=\exp\bigg(\sum_{n\geq 1}\frac{x^{ne_i}}{n^2}\bigg) \in \Hat{G}.\]
This scattering diagram is unique up to equivalence.\qed \end{thm}

 A simple calculation shows   $\Phi_\DD(\dd_i)$ acts on $\Hat{B}$ via the cluster transformation
\begin{equation}\label{clu} z^m\mapsto z^m\cdot (1+x^{e_i})^{m(e_i)}.\end{equation}
We call this scattering diagram the \emph{cluster scattering diagram} associated to the quiver $Q$.

\subsection{}
\label{genteel}
To prove Theorem \ref{same} we must give a categorical description of incoming walls.
\begin{defn}
We say that an object $E\in \A$ is \emph{self-stable} if it is $\theta$-stable for the weight  \[\theta=\<-,E\>\in M.\]
An equivalent condition is that for every non-trivial short exact sequence
\[0\lra A\lra E\lra B\lra 0\]
the inequality $\<A,B\>< 0$ holds. 
\end{defn}

We will call a quiver with potential $(Q,W)$  \emph{genteel} if the only self-stable objects in $\A$ are of the form $S_i$ for some vertex $i\in V(Q)$.


\begin{lemma}
\label{wrong}
If $(Q,W)$ is genteel then 
  the stability scattering diagram is equivalent to the cluster scattering diagram.\footnote{The proof of this Lemma is false, and the  claim is most probably incorrect. See Section \ref{erratum} below.}
\end{lemma}

\begin{proof}
For any quiver with potential let us use \eqref{climbthissummer?} to compute the wall-crossing automorphism in the stability scattering diagram at a generic point of the wall $e_i^\perp$. 
Fix $m\in M^\oplus$ and set $d=k\cdot e_i$. The extended quiver $Q^\star$ corresponding to $m$ has $m_i=m(e_i)$ extra arrows from the extended vertex $\star$ to the vertex $i$. 
At a general point $\theta\in e_i^\perp$ it is then easy to see that the framed moduli space $F(d,m,\theta)$ is isomorphic to the Grassmannian of $d$-dimensional quotients of an $m_i$-dimensional space.
Thus  we have
\[\Phi_{\dd_i}(z^m)=z^m\cdot \sum_{k\geq 0} e(\operatorname{Gr}_{k,m_i}) x^{k e_i}=z^m\cdot \sum_{k=0}^{m_i} {m_i \choose k} x^{k e_i}=z^m\cdot (1+x^{e_i})^{m_i},\]
exactly as in \eqref{clu}.  If $(Q,W)$ is genteel then by definition there are no  other incoming walls.
\end{proof}

 \subsection{}
 The proof of Theorem \ref{same} is completed by the following simple lemma. 

\begin{lemma}
Suppose $Q$ is acyclic and hence $W=0$. Then $(Q,W)$ is genteel.
\end{lemma}

\begin{proof}
The acyclic assumption means that we can label the vertices of $Q$  in such a way that
\[i<j \implies a_{ji}=\dim_\C \Ext^1_\A(S_j,S_i)=0.\]
In particular this means that if $i<j$ then $\<e_i,e_j\>=a_{ji}-a_{ij}<0$.
Consider a stability function $Z$ such that $\phi(S_i)<\phi(S_j)$ whenever $i<j$. Then it is easy to see by induction on the total dimension that the only semistable objects are of the form $S_i^{\oplus m}$.

Clearly any simple object is self-stable since it has no proper subobjects. Suppose $0\neq E\in \A$ is indecomposable but  not simple. Then $E$ has a non-trivial Harder-Narasimhan filtration with respect to the stability function $Z$, which means we can find $1\leq k<n$ and a short exact sequence \[0\lra A \lra E \lra B\lra  0\] such that the stable factors of $A$ are a subset of the simples $S_{k+1},\cdots, S_n$, and the stable factors of $B$ are a subset of the simples $S_{1}, \cdots, S_k$. This implies that $\<B,A\><0$ and hence $E$ is not self-stable.
\end{proof} 

It seems an interesting question to determine which quivers  have genteel potentials. One can check for example that the triangular quiver with potential $W=abc$ is genteel. Note though that this  is mutation-equivalent to an acyclic quiver.

\subsection{Erratum added May 2020}\label{erratum}The author is grateful to Bernhard Keller and Lang Mou for pointing out an error in the proof of Lemma \ref{wrong}, and apologises for any confusion caused. The problem is with the last sentence of the proof. Let $\dd$ be an incoming  wall of the stability scattering diagram, and 
suppose there exists a  $\theta$-stable object $E$ for some element $\theta\in \dd$. When $\operatorname{rank}(N)=2$ it follows from the non-degeneracy of the form $\<-,-\>$ that $\theta$ is a multiple of $\<-,E\>$, and hence $E$ is self-stable. But of course this is false in general, and  there is therefore no reason to expect Lemma \ref{wrong} to hold.

Fortunately, Theorem \ref{same} from the Introduction is nonetheless correct. Suppose $Q$ is an acyclic quiver, and let us label the vertices  as in the proof of Lemma 11.5, so that if an arrow connects a vertex $i$ to a vertex $j$ then $i<j$. Then for any $\theta\in M_{\R}$ satisfying the ordering condition $i<j\implies  \theta(e_i)<\theta(e_j),$  it is easy to see that the only possible $\theta$-stable objects are the vertex simples $S_i$. Since this ordering condition is unaffected by the  addition of elements of $M_\R$,  we can connect the chambers $M_{\R}^\pm$ by a path in $M_{\R}$ all of whose points satisfy it. It follows that the only walls we encounter when following this path in the stability scattering diagram are contained in the hyperplanes $e_i^{\perp}$, and these are described explicitly   by the (correct part) of the proof  of Lemma \ref{wrong}. In particular,  for this special ordered path, the walls and wall-crossing automorphisms encountered in the stability and cluster scattering diagrams are the same. It follows that the two scattering diagrams are equivalent.

The author would also like to point out \cite[Conjecture 3.3.4]{ks3} which he was unfortunately unaware of when the original paper was written. For more recent results on the relationship between the stability and cluster scattering diagrams, see 
{\small 
\begin{itemize}
\item Fan Qin  ``Bases for upper cluster algebras and tropical points", arXiv:1902.09507, 
\item Lang Mou, ``Scattering diagrams of quivers with potential and mutations", arXiv:1910.13714.
\end{itemize}}


\end{document}